\documentclass[11pt,reqno,a4paper]{extarticle}

\usepackage[osf]{newtxtext}

\usepackage[left=3cm,right=3cm,top=2.2cm,bottom=2.2cm]{geometry}

\usepackage{graphicx, varioref, amsfonts}
\usepackage{amsthm,amssymb,amsmath,mathrsfs,mathtools,stmaryrd} %\usepackage{stmaryrd}
\usepackage{dsfont}
\usepackage{upgreek} %% txgreeks
\usepackage{esint} %% nice integral symbols 
\usepackage{breakurl}
\usepackage[full]{textcomp}
\usepackage{empheq}
\usepackage{enumerate,enumitem}
\usepackage[hidelinks]{hyperref}
\usepackage{soul}
\usepackage{cancel}
\usepackage{amscd}
\usepackage{srcltx}
\usepackage{ifthen}
\usepackage{float}
\usepackage{color}
\usepackage{comment}
\usepackage{chngcntr}
\usepackage{etoolbox}

\usepackage{titlefoot}

 \usepackage{authblk} % For \affil command
\newenvironment{nalign}{
	\begin{equation}
	\begin{aligned}
}{
	\end{aligned}
	\end{equation}
	\ignorespacesafterend
}
\theoremstyle{plain}
\newtheorem{lemma}{Lemma}[section]

\newtheorem{corollary}{Corollary}[section]
\newtheorem{theorem}{Theorem}[section]
\newtheorem{assumption}{Assumption}[section]
\newtheorem{algorithm}{Algorithm}[section]

\theoremstyle{definition}

\newtheorem{example}{Example}[section]
\newtheorem{remark}{Remark}[section]
\newcommand{\R}{\ensuremath{\mathbb R}} 

\newcommand{\p}{\mathbb{P}}
\newcommand{\e}{\mathbb{E}}
\newcommand{\A}{{\mathbb{A}}} % admissible controls
\newcommand{\C}{{\mathcal{C}}} % convex constraint
\DeclareMathOperator*{\LL}{\mathcal{L}}
\newlist{todolist}{itemize}{2}
\setlist[todolist]{label=$\square$}
\usepackage{pifont}
\begin{document}

\title{Optimal control of mean field equations with monotone coefficients and applications in neuroscience}
\date{July 2, 2020}

\author{Antoine Hocquet}
\author{Alexander Vogler}

\affil{\small Technische Universit\"at Berlin, Berlin, Germany}

\maketitle

\unmarkedfntext{\textit{Mathematics Subject Classification (2020) ---} 93E20, 92B20, 65K10}

\unmarkedfntext{\textit{Keywords and phrases ---} Optimal control, McKean--Vlasov equations, FitzHugh-Nagumo neurons, stochatic differential equations, gradient descent}

\unmarkedfntext{\textit{Mail}: \textbullet$\,$ antoine.hocquet86@gmail.com $\,$\textbullet$\,$ vogler@math.tu-berlin.de}

% 92B20: Neural networks for/in biological studies, artificial life and related topics [Se
% 82C32 Neural nets applied to problems in time-dependent statistical mechanics
% 93E20 Optimal stochastic control
% 65K10 Numerical optimization and variational techniques

\begin{abstract}
We are interested in the optimal control problem associated with certain quadratic cost functionals depending on the solution $X=X^\alpha$ of the stochastic mean-field type evolution equation in $\mathbb R^d$
\begin{equation}
\label{star}
dX_t=b(t,X_t,\mathcal L(X_t),\alpha_t)dt+\sigma(t,X_t,\mathcal L(X_t),\alpha_t)dW_t\,,
\quad X_0\sim \mu
\enskip\text{($\mu$ given),}
\end{equation}
under assumptions that enclose a sytem of FitzHugh-Nagumo neuron networks, and where for practical
purposes the control $\alpha_t$ is deterministic. 
To do so, we assume that we are given a drift coefficient that satisfies a one-sided Lipshitz condition, and that the dynamics \eqref{star} is subject to a (convex) level set constraint of the form $\pi(X_t)\leq0$.
The mathematical treatment we propose follows the lines of the recent monograph of Carmona and Delarue for similar control problems with Lipshitz coefficients.
After addressing the existence of minimizers via
a martingale approach, we show a maximum principle for \eqref{star}, and numerically investigate a
gradient algorithm for the approximation of the optimal control.
\end{abstract}

\tableofcontents

\section{Introduction}

\subsection*{Motivations}
\label{sec:intro}

Based on a modification of a model by van der Pol, FitzHugh \cite{fitzhugh1961impulses} proposed in 1961 the following system of equations in order to describe the dynamics of a single neuron subject to an external current $I$:
\begin{equation}
\label{single_neuron}
\begin{aligned}
&\dot v= v - \frac{1}{3}v^3 -w +I
\\
& \dot w=c(v +a -bw)
\end{aligned}
\end{equation}
for some constants $a,b,c>0$,
where the unknowns $v,w$ correspond respectively to the so-called voltage and recovery variables (see also Nagumo \cite{nagumo1962active}).
In presence of interactions, one has to enlarge the previous pair by an additional unknown $y$ that counts a fraction of open channels (synapic channels), and which is sometimes referred to as gating variable.

When it comes to an interacting network of neurons, it is customary to assume that the corresponding graph is fully connected, which is arguably a good approximation at small scales \cite{SchDeg17}. This implies that all the neurons in the given network add a contribution to the interaction terms in the equation.
Precisely, for a population of size $N\in\mathbb N,$  the state at time $t$ of the $i$-th neuron is described by the three-dimensional vector 
\[ X^i_t=(v^i_t,w^i_t,y^i_t) ,\quad i=1,\dots N,\]
and one is led to study the system of $3N$ stochastic differential equations:
\begin{equation}
\label{FHN}
\left\{
\begin{aligned}
&dv^i_t=\Big(v^i_t-\frac{(v^ i_t)^3}{3}-w_t^i + I_t\Big)dt  +\sigma _{ext} d W^i_t
%^{\text{Local dynamics}}
\\
&\quad\quad  
- \frac{1}{N}\sum\nolimits_{j=1}^N\bar J(v^i_t-V_{rev})y^j_tdt
	- \frac{1}{N}\sum\nolimits_{j=1}^N\sigma^J(v^i_t-V_{rev})y^j_tdB_t^i
%{\text{Interactions}}
\\
&dw^i_t=c (v^i_t+a -b w^i_t)dt\enskip ,
\\
&dy^i_t=(a_r S (v_t^i)(1-y_t^i)-a_d y_t^i)dt
+
% \chi(y^i_t)\sqrt{\overline{a}S(v^i_t)(1-y^i_t)+\overline{b}y^i_t}
\sigma^{y^i}(v^i)d\tilde B^i_t\,.
\end{aligned}\right.
\end{equation}
In the above formula, $B^{i}$, $W^i$, $\tilde B^i$ are i.i.d.\ Brownian motions modelling independent sources of noise with respective intensities $\sigma^J,\sigma_{ext},\sigma^{y^i}(v^i)>0$. The last of these intensities depends on the solution, through the formula 
 \begin{equation}\label{sigma_y}
 \sigma^{y}(v)=\chi(y)\sqrt{\overline{a}S(v)(1-y)+\overline{b}y}
 \end{equation}
  with given constants $\overline{a},\overline{b}>0$ and
some smooth cut-off function $\chi\colon\R\to\R$ supported in $(0,1).$ 
Various physical constants appear in \eqref{FHN}, which we now briefly introduce: 
\begin{itemize}
	\item $V_{rev}$ is the synaptic reversal potential;
	\item $\bar J$ is (the mean of) the maximum conductance; % $J$;
	\item $S(v^i)$ is the concentration of neurotransmitters released into the synaptic cleft by the presynaptic neuron $i$;
	explicitly for $v\in\R$
	\begin{equation}
	\label{S_v}
	S(v)=\dfrac{T_{max}}{1+e^{-\lambda(v-V_T)}}
	\end{equation}
where $T_{max}$ is a given maximal concentration and $\lambda^{-1}>0,V_T>0,$ are constants setting the steepness, resp.\ the value,
% gives a steepness and $V_T$ sets the value 
at which $S(v)$ is half-activated (for typical values, see for instance \cite{destexhe1994synthesis});
	\item $a_r,a_d>0$ correspond to rise and decay rates, respectively, for the synaptic conductance.
	
%   released into the synaptic cleft by a presynaptic spike;
\end{itemize}
In this model, the voltage variable $v^i$ is describing the membrane potential of the $i$-th neuron in the network, while the recovery variable $w^i$ is modeling the dynamics of the corresponding ion channels.
As already alluded to, the gating variable $y^i$ models a \textit{fraction} of open ion channels in the postsynaptic neurons, and thus ought to be a number between $0$ and $1$ (hence the cut-off $\chi(y^i)$ in \eqref{sigma_y}). Loosely speaking, $y^i$ should be thought as the output contribution of the neuron $i$ to adjoining postsynaptic neurons, resulting from the concentration $S(v^i)$ of neurotransmitters.
%In this picture, the neuron $i$ will release neurotransmitters into the synaptic cleft between the presynaptic neuron and the postsynaptic neuron, which in turn will bind to receptors in the postsynaptic cell.
%Thus neuron $i$ induces the opening of a fraction $y^i$ of ion channels at the dendrites of the postsynaptic cell. 
The resulting synaptic current from $i$ to $j$ affecting the postsynaptic neuron $j$ is then given by $-J(v^j-V_{rev})y^i$ where $J$ is the maximum conductance. This latter term is affected by noise coming from the environment, which in turn explains the structure of the interaction terms in the first equation. 
%As is the case of the external current,
%Of course the fraction $y^i$ is covered by the dynamic of ion channels as a result of an diffusion approximation. 
%Thus it is affected by channel noise.
For a thorough presentation of \eqref{FHN} and its applications in the field of neurosciences, we refer for instance to the monograph of Ermentrout and Terman \cite{ermentrout2010mathematical}.

\subsection*{Propagation of chaos}
%For a large number of neurons, 
The system \eqref{FHN} has the generic form 
\begin{equation}\label{pre_mckean}
\left \{\begin{aligned}
dX_t^{N,i}&=b(t,X_t^{N,i},\overline{\mu}_{X_{t}^N},\alpha_t)dt+\sigma(t,X_t^{N,i},\overline{\mu}_{X_{t}^N},\alpha_t)dW_t^i\,, & t\in [0,T],
\\
X_0^{N,i}&\sim u_0,
\end{aligned}\right .
\end{equation}
for $i=1,\dots,N$, where $u_0$ is a probability measure on $\R^d$,
$(\alpha_t)$ is a control and $\bar\mu_{X_t^N}$ denotes the empirical measure
\begin{align*}
\overline{\mu}_{X_{t}^N}:=\dfrac{1}{N}\sum_{k=1}^{N}\delta_{X_t^{N,k}}.
\end{align*}
For $N\to \infty$, one is naturally pushed to investigate the convergence in law of the solutions of \eqref{pre_mckean} towards the probability measure $\mu=\LL(X|\p)$, where $X$ solves
\begin{equation}
\label{state}
\left \{\begin{aligned}
&dX_t
=b(t,X_t,\LL(X_t),\alpha_t)dt+\sigma(t,X_t,\LL(X_t),\alpha_t)dW_t,\quad  t\in [0,T]
\\
&X_0\in L^2(\Omega,\mathcal{F}_0,\p;\R^d).
\end{aligned}\right .
\end{equation} 
and where $b,\sigma$ are the coefficients obtained by substituting expectations in \eqref{pre_mckean} in place of empirical means.
In the context of \eqref{FHN}, a first mathematical investigation of such convergence is due to Baladron, Fasoli, Faugeras and Touboul \cite{baladron2012mean} (see also the clarification notes \cite{bossy2015clarification}).  
% For each $N\in \mathbb{N}$, \eqref{pre_mckean} has a unique strong solution $(X^{N,1},\dots,X^{N,N})$ which has continuous trajectories.
In this direction, the authors show that the sequence of symmetric probability measures 
\[ \mu_N:=\LL((X^{N,1},\dots,X^{N,N})|\p)\]
is \textit{$\mu$-chaotic}. Namely, for each $k\in \mathbb{N},k\leq N$ and $\phi_1,\dots,\phi_k\in C_b(C([0,T];\R^{d}))$
it holds
\begin{align*}
\lim\limits_{N\rightarrow \infty}\langle \mu_N,\phi_1\otimes\dots\otimes\phi_k\otimes1\otimes\dots\otimes1\rangle=\prod_{i=1}^{k}\langle \mu,\phi_i\rangle.
\end{align*}
This situation is usually referred to as ``propagation of chaos''.\\

\subsection*{Mean-field limit and control}
In this regard, taking $N\gg1$ guarantees that a ``good enough'' approximation of \eqref{FHN} is given by the mean-field limit \eqref{state},
where the corresponding coefficients
$(b,\sigma)\colon[0,T]\times \R^3\times \mathcal{P}(\R^3)\times \R\rightarrow \R^3\times \R^{3\times 3}$,
are given by
% mean field terms in \eqref{state}
\begin{equation}\label{b_FHN}
\begin{aligned}
b(t,x,\mu,\alpha)
&=\begin{pmatrix}
v-\frac{v^3}{3}-w+\alpha \\
c(v+a-bw) \\
\overline{a}S(v)(1-y)-\overline{b}y
\end{pmatrix}
+
\begin{pmatrix}
-J(v-V_{rev})\int_{\R^3}z_3\mu(dz) \\
0 \\
0
\end{pmatrix},
\end{aligned}
\end{equation}
for $x=(v,w,y),$ and
\begin{equation}\label{sigma_FHN}
\begin{aligned}
\sigma(t,x,\mu,\alpha)&=\begin{pmatrix}
\sigma_{ext} & -\sigma^J(v-V_{rev})\int_{\R^3}z_3\mu(dz) & 0 \\
0 & 0 & 0 \\
0 & 0 & \chi(y)\sqrt{\overline{a}S(v)(1-y)+\overline{b}y}
\end{pmatrix}.
\end{aligned}
\end{equation}

In this paper, we concentrate our attention on the optimal control problem associated with a cost functional of the form
\begin{equation}
\label{cost_fun}
J:\A\rightarrow \R,\quad 
\alpha\mapsto \e\left(\int_{0}^{T}f(t,X_t^\alpha,\LL(X_t^\alpha),\alpha_t)dt+g(X_T^\alpha,\LL(X_T^\alpha))\right),
\end{equation} 
for suitable functions $f$ and $g$, and where $X^\alpha$ is subject to the dynamical constraint \eqref{state}. The functional cost ought to be minimized over some convex, admissible set of controls $\A.$

Because of potential applications in the treatment of neuronal diseases, the control of the stochastic FHN model has gained a lot of attention during the last years (see, e.g., \cite{stochFHN,barbu2016optimal}). 
The need to introduce random perturbations in the original model is widely justified from a physics perspective (see for instance \cite{deco2009stochastic} and the references therein).
In \cite{stochFHN} the authors investigate a FitzHugh-Nagumo SPDE which results from the continuum limit of a network of coupled FitzHugh-Nagumo equations.  
We have a similar structure in mind regarding the dependence of the coefficients on the control (namely, the dynamics of the membrane potential depends linearly on the control). Our approach here is however completely different, in that we hinge on the McKean-Vlasov type SDE \eqref{state} that originates from the propagation of chaos.

McKean-Vlasov control problems of this type were investigated in the past decade by Bensoussan, Frehse and Yam \cite{bensoussan2013mean}, but also by Carmona and co-authors (see for instance \cite{carmona2013control}).
These developments culminated with the monograph of Carmona and Delarue \cite{carmona2018probabilistic}, where a systematic treatment is made (under reasonable assumptions).
Other related works include \cite{bonnans2019schauder,albi2017mean,buckdahn2016stochastic,dos2019freidlin}.
These results fail however to encompass \eqref{state}--\eqref{sigma_FHN}, due for instance to the lack of Lipshitz property for the drift coefficient.

From the analytic point of view, the FitzHugh-Nagumo model also suffers the fact that the diffusion matrix is degenerate, making difficult to obtain energy estimates for the Kolmogorov equation (see Remark \ref{rem:FP}).

%Difficulties arise  from the fact that the coefficients are only locally Lipshitz with respect to $X$ (with some monotonicity condition). 

% Mass concentration for artificial neural networks
% infinite dimensional statistics

Our objective in this work is twofold.
At first, our purpose is to extract some of the qualitative features of FitzHugh-Nagumo system and its mean field limit, in a broader treatment that encloses \eqref{FHN} and \eqref{state}--\eqref{sigma_FHN}. In this sense, our intention is not to deal with the previous models ``as such'' but instead, we aim to take a step further by dealing with a certain class of equations that possess the following attributes:
\begin{itemize}
	\item (Monotonicity) -- though the drift coefficient in \eqref{state} displays a cubic non-linearity, it satisfies the monotonicity condition $\langle x-x',b(t,x,\mu,\alpha)-b(t,x',\mu,\alpha)\rangle\lesssim |x-x'|^2$.
	\item (Constrained dynamics) -- the dynamics of the coupling variable ensures that the convex constraint $y_t\in[0,1]$ holds for all times.
	\item (Interaction with quadratic dependence on the unknown) -- In spite of the order 1 type interaction in \eqref{b_FHN}-\eqref{sigma_FHN} (in the sense of \cite[p.~134]{carmona2013control}), the corresponding nonlinearity displays the quadratic behaviour $|b(t,x,\mu,\alpha)-b(t,x,\nu,\alpha)|\lesssim (1+|x|^2)W_2(\mu,\nu)$.
\end{itemize}
Under the above setting, we aim to develop and implement direct variational methods, in the spirit of the stochastic approach of Yong and Zhou \cite{yong1999stochastic} for classical control problems (note that some work in this direction has been already done by Pfeiffer \cite{pfeiffer2015optimality,pfeiffer2017numerical}, in a slightly different setting). 
Second, we aim to derive a Pontryagin maximum principle for mean-field type control problems of the previous form, with a view towards efficient numerical approximations of optimal controls (e.g.\ gradient descent).

\subsection*{Organization of the paper}
In Section \ref{sec:preliminaries} we introduce our assumptions on the coefficients and give the main results.
Section \ref{sec:well} is devoted to the well-posedness of the main optimal control problem (Theorem \ref{thm:ex}). In Section \ref{sec:max}, we show the corresponding maximum principle (Theorem \ref{thm:max}).
Finally, Section \ref{sec:numerics} will be devoted to numerical examples.

\section{Preliminaries}
\label{sec:preliminaries}
\subsection{Notation and settings}
In the whole manuscript, we consider an arbitrary but finite time horizon $T>0$. 
We fix a dimension $d\geq 1$, and denote the scalar product in $\R^d$ by $\langle\cdot,\cdot \rangle.$
If $A,B$ are matrices of the same size, we shall also write $\langle A, B\rangle$ for their scalar product, namely
\[
\langle A,B\rangle:= \mathrm{tr}(A^{\dagger}B)
\]
where $A^\dagger$ is the transposed matrix, and $\mathrm{tr}$ the trace operator.
For a continuously differentiable function $f\colon \R^d\to \R$, we adopt the suggestive notation $f_x$ to denote its Jacobian (seen for each $x\in \R^d$ as an element of the dual of $\R^d$). Given $h\in \R^d$, we let 
\begin{equation}\label{differential}
f_x(x)\cdot h
\end{equation}
be the evaluation of $f_x(x)$ at $h.$  A similar convention will be used for vector-valued functions.

Throughout the paper, we fix a complete filtered probability space $(\Omega,\mathcal{F},(\mathcal{F}_t)_{t\in [0,T]},\p)$ carrying an $m$-dimensional Wiener process $(W_t)_{t\in [0,T]}$. 
Given $p\in[1,\infty)$ and a $p$-integrable random variable $X$, we denote its usual $L^p$-norm by $\|X\|_p:=\e(|X|^p)^{1/p}$.
%(with the standard convention when $p=\infty$).
We further introduce the spaces
\begin{align*}
\mathcal{H}^{2,d}&:=\bigg\lbrace Z:\Omega\times [0,T]\rightarrow \R^d\,\bigg|\, Z\text{ prog.\ measurable and  }\int_{0}^{T}\|Z_t\|_2^2dt<\infty\bigg\rbrace
\\
\mathcal{S}^{2,d}&:=\bigg\lbrace Z:\Omega\times [0,T]\rightarrow \R^d\,\bigg|\, Z\text{ prog.\ measurable, continuous and  }\Big\|\sup_{t\in [0,T]}|Z_t|\Big\|^2_2<\infty\bigg\rbrace.
\end{align*}
For $m\in\mathbb N$, the notations $\mathcal{S}^{2,d\times m}$, $\mathcal{H}^{2,d\times m}$ will also be used to denote the corresponding sets of $d\times m$ matrix-valued processes.
Whenever clear from the context, we will omit to indicate dimensions and write $\mathcal{S}^2$ or $\mathcal{H}^2$ instead.

We will denote by $\mathcal{P}(\R^d)$ the set of all probability measures on $(\R^d,\mathcal{B}(\R^d))$. For $p\in[1,\infty)$, $\mu\in \mathcal{P}(\R^d)$ we define the moment of order $p$:
\begin{align*}
\mathcal{M}_p(\mu)^p:=\int_{\R^d}|x|^p\mu(dx)\in [0,\infty],
\end{align*}
and we let $\mathcal{P}_p(\R^d):=\left\lbrace \mu\in \mathcal{P}(\R^d)\,\big|\,\mathcal M_p(\mu)<\infty\right\rbrace.$
By $W_p,$ $p\in[1,\infty),$ we denote the usual $p$-Wasserstein distance on $\mathcal P_p$, that is for $\mu,\nu\in\mathcal P_p(\R^d)$
\begin{equation}\label{def:wasserstein}
W_p(\mu,\nu)^p:=\inf _{\pi\in\Pi(\mu,\nu)}\iint_{\R^d\times\R^d}|x-y|_{\R^d}^p\pi(dx\times dy),
\end{equation}
where $\Pi(\mu,\nu)$ denotes the set of probability measures on $\R^d\times\R^d$ with $\mu$ and $\nu$ as respective first and second marginals (we refer to \cite[Chap.~5]{carmona2018probabilistic} for a thorough introduction to the subject).
Moreover, we recall the following elementary but useful consequence of the previous definition.
Let $\mu,\nu$ be in $\mathcal P_p,$ and assume that there are random variables $X,Y$ on $(\Omega,\mathcal F,\p)$ such that $X\sim \mu$ and $Y\sim\nu.$ Then, it holds 
\begin{equation}\label{elementary_W2}
W_p(\mu,\nu)\leq \e\left(|X-Y|^p\right)^\frac{1}{p}.
\end{equation}
Finally, whenever $f\colon \mathcal P_2\to \R$ is continuously L-differentiable at some $\mu\in\mathcal P_2$,
we write $f_\mu(\mu)(x)$ to denote its Lions derivative at the point $(\mu,x)\in \mathcal P_2\times\R^d$. 
In keeping with the notation \eqref{differential} on differentials, we will let
$f_\mu(\nu)(x)\cdot h$ be its evaluation (as an element of the dual of $\R^d$) at $h\in\R^d$.

\subsection{Controlled dynamics and cost functional}
\label{sec:cost}
Our controlled dynamics will be given by a McKean-Vlasov type SDE (state equation) of the form \eqref{state}, where $X_0\in L^r(\Omega,\mathcal{F}_0,\p;\R^d)$ for some fixed $r\geq 6$ and $\alpha$ is an admissible control, i.e.\ for some convex set $A\subset \R^k$ and some constant $K>0$ fixed throughout the paper, 
\begin{align}
\label{adm_cont}
\alpha\in \A:= \left \{\alpha:[0,T]\rightarrow A\enskip\Bigg|\enskip\int_{0}^{T}|\alpha(t)|^rdt\leq K \right \}.
\end{align}
In the whole manuscript, we assume that we are given continuous running and terminal cost functions
 \begin{align*}
 f&\colon[0,T]\times \R^d\times \mathcal{P}_2(\R^d)\times A\rightarrow \R\\
 g&\colon\R^d\times \mathcal{P}_2(\R^d)\rightarrow \R
 \end{align*}
which have quadratic growth in the following sense:
there exists $C>0$ such that for all $t\in [0,T]$, $x\in \R^d$, $\alpha\in A$ and $\mu\in \mathcal{P}_2(\R^d)$
	\begin{align*}
	|f(t,x,\mu,\alpha)|&\leq C(1+|x|+ \mathcal{M}_2(\mu)+|\alpha|)^2\\
	|g(x,\mu)|&\leq C(1+|x|+\mathcal{M}_2(\mu)))^2.
	\end{align*}
We will then consider the cost functional 
\begin{nalign}
	J\colon\A\rightarrow \R,\quad\alpha\mapsto \e\left(\int_{0}^{T}f(t,X_t^\alpha,\LL(X_t^\alpha),\alpha_t)dt+g(X_T^\alpha,\LL(X_T^\alpha))\right).
\end{nalign}

%%%
%%%
\subsection{Level set constraint}
\label{sec:level}

A formal application of It\^o Formula reveals that the constraint
\[\C:= \left \{x=(v,w,y):0\leq y\leq 1\right \} \]
 is preserved along the flow of the state equation associated with a network of FitzHugh-Nagumo neurons. This is of course coherent with the intuition that $y$ is a fraction of open channels. In other words, we have $\pi(X)\leq 0$ where $\pi\colon \R^3\to \R,$ is the map $x\mapsto y(y-1).$ 
Motivated by this example, we will assume in the sequel that we are given a convex function $\pi\in C^2(\R^d,\R)$ such that any solution $X$ is supported in $\C\subset\R^d$ for all times, where $\C$ is the set
\begin{equation}\label{convex_set}
\C:=\left \{x\in \R^d:\pi(x)\leq 0\right \}.
\end{equation}
We suppose moreover that $\C$ contains at least one element, which for convenience is assumed to be $0$.
To ensure that the constraint is preserved, we need to assume that $\pi(X_0)\leq 0$, $\p\text{-almost surely}$. Furthermore we need to make the following compatibility assumptions on $\pi\colon \R^d\to \R.$
\begin{assumption}[constrained dynamics]
\label{H_pi}
For all $\mu\in \mathcal{P}(\R^d),\alpha\in A,$ $t\in [0,T]$ and $x\in \R^d\setminus\C$,
we have 
\begin{align}
\label{pi_b}
&\pi_x(x)\cdot b(t,x,\mu,\alpha)\leq 0,
\intertext{while}
\label{pi_sigma}
&\mathrm{Im}\left (\sigma(t,x,\mu,\alpha)\right )\subset 
\pi_x(x)^{\perp}
\quad \text{and}\quad 
\pi_{xx}(x)\cdot(\sigma\sigma^\dagger(t,x,\mu,\alpha)) = 0.
\end{align}
\end{assumption}
 
\begin{example}[Gating variable constraint for FitzHugh-Nagumo]
Assumption \ref{H_pi} is fulfilled for \eqref{state}--\eqref{sigma_FHN} and with $\pi(v,w,y)=y(y-1)$, as can be seen as follows. We have the identities (recall the notation \eqref{sigma_y})
\[
\begin{aligned}
&\pi_x (x)= 
	\begin{pmatrix}
	0 & 0 & 2y-1
	\end{pmatrix},\quad 
\pi_{xx}(x)=
	\begin{pmatrix}
	0 &0 & 0
	\\
	0 & 0 & 0
	\\
	0 & 0 & 2
	\end{pmatrix},
\\
&
\sigma\sigma^{\dagger}(t,x,\mu,\alpha)= 
	\begin{pmatrix}
\sigma_{ext}^2 + (\sigma^J)^2(v-V_{rev})^2(\int_{\R^3}z_3\mu(dz))^2 
&0 & 0
\\
0 & 0 & 0
\\
0 & 0 & \sigma^y(v)^2
\end{pmatrix},
\end{aligned}
\]
Clearly, \[
\pi_x(x)\sigma(t,x,\mu,\alpha) \in \mathrm{Linspan}\left \{\begin{pmatrix}0 & 0 & (2y-1)\sigma^y(v)\end{pmatrix}\right \}\,.
\]
 But using $\mathrm{Supp}\chi\subset(0,1)$, we find indeed that
$(2y-1)\sigma^y(v)=0$ outside $\C$. The same argument implies
\[
\pi_{xx}(x)\cdot (\sigma\sigma^\dagger(t,x,\mu,\alpha)) = 2\sigma^y(v)^2
\]
and the latter vanishes if $x\notin\C$, hence \eqref{pi_sigma}.

Towards \eqref{pi_b}, one observes letting $q=\bar a S(v)$ that
\[
\pi_x(x)\cdot b(t,x,\mu,\alpha)
= -q +(3q + \bar b)y -2(q+\bar b)y^2 = P(y)\,.
\]
The polynomial $P(y)$ has discriminant $(q-b)^2$, hence the roots
\[
r_- = \frac{q}{q+\bar b}\,,\quad r_+=\frac12\,,
\]
which both lie in the interval $(0,1)$. It follows that $P(y)$ is negative outside $\C$, implying \eqref{pi_b}.
\end{example}

\subsection{Regularity assumptions and main results}

Besides Assumption \ref{H_pi}, one needs to make suitable hypotheses on the regularity of the drift and diffusion coefficients.
In the sequel, we denote by $\mathcal P_2^{\C}(\R^d)$ the subset of all probability measures in $\mathcal P_2(\R^d)$ which are supported in $\C=\pi^{-1}((-\infty,0]).$

\begin{assumption}[MKV Regularity] 
	\label{ass:mkv}
	We assume that the coefficients
	\begin{align*}
	(b,\sigma):[0,T]\times \R^d\times \mathcal{P}_2(\R^d)\times A\rightarrow \R^d\times \R^{d\times m}
	\end{align*}
	 are locally Lipshitz. Moreover, there are constants $L_1,L_2,L_3>0$ such that the following properties hold.
\begin{enumerate}[label=(L\arabic*)]
\item\label{L1} -- (regularity of the diffusion coefficient) --
	The diffusion coefficient $\sigma$ satisfies the property
	$\sup_{0\leq t\leq T}|\sigma(t,0,\delta_0,0)|^2<\infty.$
Moreover, for all $t\in [0,T]$, $x\in \R^d$, $\alpha\in A$ and $\mu\in \mathcal{P}^{\C}_2(\R^d)$ we have
\begin{equation}\label{hyp:sigma1}
 |\sigma(t,x,\mu,\alpha)|^2\leq L_1(1+|\alpha|^2+|x|^2) 
\end{equation}

If $x,x'\in \R^d $, $\alpha'\in A$, then
\begin{equation}\label{hyp:sigma2}
|\sigma(t,x,\mu,\alpha)-\sigma(t,x',\mu,\alpha')|^2\leq L_1(|x-x'|^2+|\alpha-\alpha'|^2).
\end{equation}
Finally, if $\in \R^d$ and $\mu'\in \mathcal{P}_2(\R^d)$, then
\begin{equation}\label{hyp:wass_sigma}
|\sigma(t,x,\mu,\alpha)-\sigma(t,x,\mu',\alpha)|^2\leq L_1(1+|x|^2)W_2(\mu,\mu')^2. 
\end{equation}
	
	\item\label{L2} -- (regularity of the drift coefficient) -- There exists $q\in \mathbb{N}$ with $4q\leq r$, such that for all $t\in [0,T]$, $x,x'\in \R^d$, $\alpha,\alpha'\in A$ and $\mu\in \mathcal{P}_2(\R^d)$
	\begin{equation}\label{hyp:loc_lip}
	\begin{aligned}
	&|b(t,x,\mu,\alpha)-b(t,x',\mu,\alpha')|
	\\
	&\leq \sqrt{L_2}(1+|x|^{q-1}+|x'|^{q-1}+|\alpha|^{q-1}+|\alpha'|^{q-1}+\mathcal{M}_2(\mu)^2)(|x-x'|+|\alpha-\alpha'|).
	\end{aligned}
	\end{equation}
	In addition, $b$ satisfies the following Lisphitz property with respect to the Wasserstein distance: 
	for all $t\in [0,T]$, $x\in \R^d$, $\alpha\in A$ and $\mu,\mu'\in \mathcal{P}_2(\R^d)$
	\begin{equation}\label{hyp:wass_b}
	|b(t,x,\mu,\alpha)-b(t,x,\mu',\alpha)|^2\leq L_2(1+|x|^2)W_2(\mu,\mu')^2. 
	\end{equation}
	
	\item\label{L3} -- (monotonicity of the drift) --
	The drift coefficient $b$ is such that 
	$\sup_{0\leq t\leq T}|b(t,0,\delta_0,0)|<\infty.$
	Moreover, for all $t\in [0,T]$, $x\in \C$, $\alpha\in A$ and $\mu\in \mathcal{P}^{\C}_2(\R^d)$
	it holds
	\begin{equation}\label{hyp:mono1}
	\langle x,b(t,x,\mu,\alpha)\rangle \leq L_3(1+|\alpha|^2+|x|^2)
	\end{equation}
	and if $x'\in \C$, $\alpha'\in \A$, then
	\begin{equation}\label{hyp:mono2}
	\left \langle x-x',b(t,x,\mu,\alpha)-b(t,x',\mu,\alpha')\right \rangle 
	\leq L_3(|x-x'|^2+|\alpha-\alpha'|^2).
	\end{equation}
%    where we are given $\rho\colon\R\rightarrow \R$ measurable and such that $|\rho|_{\infty,\C}:=\esssup_{x\in \C}\rho(x)<\infty.$
\end{enumerate}
\end{assumption}

\begin{example}[Analysis of the FitzHugh-Nagumo model]\label{exa:FHN}
	Let us go back to the settings of \eqref{state}--\eqref{sigma_FHN} for a coupled system of FitzHugh-Nagumo neurons.
	Trivially, one has $\sup_{0\leq t\leq T}|\sigma(t,0,\delta_0,0)|=|\sigma_{ext}|<\infty.$
	The map $v\mapsto S(v)$ being positive and bounded, we further see that the $(3,3)$-th entry of $\sigma$ is Lipshitz, as deduced immediately from the fact that $\chi$ is supported in $(0,1).$ For the remaining non-trivial component, we have
	\[
	\sigma^{1,2}(x,\mu,\alpha)^2 
	\leq J(V_{rev} + |v|^2) |\beta(\mu)| 
	 \]
	where to ease notation we introduce the barycenter $\beta(\mu)$, defined as the quantity
	 \begin{equation}\label{barycenter}
	 \beta(\mu):=\int_{\R^3}z_3\mu(dz_1\times dz_2\times dz_3).
	 \end{equation}
	The condition $\mathrm{Supp}\mu\subset \C,$ implies trivially that $|\beta(\mu)|\leq 1$ and thus we obtain \eqref{hyp:sigma1} for $L_1=(V_{rev}J)\vee 1.$ The Lipshitz-type property \eqref{hyp:sigma2} is shown in a similar fashion.

The Wasserstein-type regularity \eqref{hyp:wass_sigma} is hardly more problematic: using the Kantorovitch duality Theorem \cite[Prop.~5.3 \& Cor.~5.4]{carmona2018probabilistic} and the fact that the projector $z=(z_1,z_2,z_3)\mapsto z_3$ is Lipshitz, one finds that
\begin{equation}\label{estim:beta}
|\beta(\mu-\mu')|=|\int_{\R^3}z_3(\mu-\mu')(dz)|\leq W_1(\mu,\mu').
\end{equation}
hence
\[
|\sigma(x,\mu) - \sigma (x,\mu') |
\leq J|v-V_{rev}|W_1(\mu,\mu').
\]
As is classical, the $1$-Wasserstein distance $W_1(\mu,\mu')$ can be estimated by $W_2(\mu,\mu'),$ which in turn implies \eqref{hyp:wass_sigma}, and thus \ref{L1}.

%	 We will focus on the drift part of our model, since in the numerical part of this work we will consider only additive noise.
As for the drift coefficient, since $b(t,0,\delta_0,0)$ is also independent of $t$, the supremum condition in \ref{L3} is clear. Moreover, it has polynomial dependency on the variables $v,w,y$, which implies the local Lipshitz property \eqref{hyp:loc_lip} with $q=3$.
We also have 
\begin{align*}
|b(t,x,\mu,\alpha)-b(t,x,\mu',\alpha)|&\leq J|v-V_{rev}||\beta(\mu-\mu')|
\end{align*}
and we conclude by \eqref{estim:beta} that \ref{L2} holds.

%, by definition of $S(v)$ in \eqref{S_v} we have for the Euclidean norm
%\[ |b(t,0,\delta_0,0)|^2
%= (ca)^2+\left (\overline{a}\dfrac{T_\text{max}}{1+e^{\lambda V_T}}\right )^2\]
To show \eqref{hyp:mono1} and \eqref{hyp:mono2}, it is enough to prove the corresponding bounds when
$c=0=\overline b ,$
since the related contributions are affine linear in the variables. Similarly, by linearity we can let $w=\alpha=0$.
But in that case, it holds
\begin{align*}
\langle x,b(t,x,\mu,0)\rangle 
\leq v^2-\dfrac{v^4}{3} +\overline{a}S(v)(1-y)y -Jv^2\beta(\mu)+JV_{rev}v\beta(\mu)\,.
\end{align*}
Observe that, since $\mu$ is supported inside $\C$, one has in particular
$\beta(\mu)\geq0$. Consequently, the fourth term in the right hand side can be ignored,
showing \eqref{hyp:mono1} with $L_3=L_3(\overline a ,|S|_{\infty},J,V_{rev})>0.$

Similarly, if $x'=(v',0,y')\in\R^3$
\begin{align*}
&\langle x-x',b(t,x,\mu,0)-b(t,x',\mu,0)\rangle
\\
&= (v-v')^2-\frac13(v^3-v'^3)(v-v')
-J(v-v')^2\beta(\mu)
\\
&\quad \quad \quad +\overline{a}(1-y)(y-y')(S(v)-S(v'))-\overline{a}S(v')(y-y')^2
\\
&\leq |S'|_{\infty}(1\vee \overline a)(1+y^2)(|y-y'|^2 + |v-v'|^2)\,.
\end{align*}
It follows that \eqref{hyp:mono1} holds with  $L_3=L_3(\overline a,\overline b,c,|S|_{C^1})>0.$
\end{example}

\begin{assumption}[Weak continuity]
	\label{ass:weak}
	For any $t\in [0,T]$, $x\in \R^d$ and $\mu\in \mathcal{P}_2(\R^d)$, the functions
	$A\to \R^d\times \R^{d\times m}\times \R$,
	$\alpha\mapsto (b,\sigma,f)(t,x,\mu,\alpha)$
	are convex.
	Furthermore, for all $x\in C([0,T];\R^d)$ and $\mu\in C([0,T];\mathcal{P}_2^\C (\R^d))$ the functions
	\begin{align*}
	\mathbb{A}\rightarrow L^2([0,T];\R^d),\quad \alpha\mapsto  b(\cdot,x_\cdot,\mu_\cdot,\alpha_\cdot),\\
	\mathbb{A}\rightarrow L^2([0,T];\R^{d\times m}),\quad\alpha\mapsto \sigma(\cdot,x_\cdot,\mu_\cdot,\alpha_\cdot),
	\end{align*}
	are weakly sequential continuous. 	
\end{assumption}

\begin{remark}\label{rem:lower}
	The continuity and convexity of $f(t,x,\mu,\cdot)$ leads to weak lower semicontinuity of the map
	\begin{align*}
	\mathbb{A}\rightarrow \R,\quad \alpha\mapsto \int_{0}^{T}f(t,x_t,\mu_t,\alpha_t)dt,
	\end{align*} 
	for all $x\in C([0,T];\R^d)$ and $\mu\in C([0,T];\mathcal{P}_2(\R^d))$.
\end{remark}
%%%

We can now present our main results.
At first, we investigate the existence of an optimal control for the following problem
\begin{align}\label{SM}\tag{SM}
\min_{\alpha\in \A }J(\alpha),
\end{align}
subject to 
\begin{equation}\label{eqa}
\left \{\begin{aligned}
	dX_t&=b(t,X_t,\LL(X_t),\alpha_t)dt+\sigma(t,X_t,\LL(X_t),\alpha_t)dW_t\,, & t\in [0,T],
	\\
	X_0&\in L^r(\Omega,\mathcal{F}_0,\p;\R^d).
\end{aligned}\right .
\end{equation}

\begin{theorem}\label{thm:ex}
	Under assumptions \ref{H_pi}--\ref{ass:weak}, the problem \eqref{SM} is finite and has an optimal control. Namely, $\inf_{\alpha\in \A }J(\alpha)<\infty$ and there is $\overline{\alpha}\in\A $, such that 
	\begin{align*}
	J(\overline\alpha)\leq J(\alpha),
	\end{align*}
	for all $\alpha\in \A $.
\end{theorem}

In order to address the corresponding maximum principle, we now introduce further assumptions on our coefficients.

\begin{assumption}[Pontryagin Principle]
	\label{ass:pontryagin}
	 The coefficients $b,\sigma,f$ and $g$ are continuously differentiable with respect to $(x,\alpha)$ %or only with respect to $x$ respectively
	 and continuously $\mathrm{L}$-differentiable with respect to $\mu\in \mathcal P_2(\R^d).$
	 Furthermore there exist $A_1,A_2,A_3>0$ such that:
	\begin{enumerate}[label=(A\arabic*)]
		\item\label{A1} For every $(s,x,\mu ,\alpha )\in [0,T]\times \C\times \mathcal P_2^\C (\R^d)\times A,$ and each $y,z\in \R^d$:
		\begin{align*}
		\langle b_x(t,x,\mu ,\alpha )\cdot z,z\rangle&\leq A_1|z|^2,\\
		|b_x(t,x,\mu,\alpha)|&\leq A_1(1+|x|^{q-1}),\\
		|b_\alpha(t,x,\mu,\alpha)|&\leq A_1,\\
		|b_\mu(t,x,\mu,\alpha)(y)|&\leq A_1(1+|x|),
		%\\
		%\sigma_\mu(t,x,\mu,\alpha)(y)&=\partial_\mu\sigma(t,x,\mu,\alpha)(y').
		\end{align*}
		where $q$ is the same constant as in \ref{L2}.
		\item\label{A2} For every $(s,x,\mu ,\alpha )\in [0,T]\times \C\times \mathcal P_2^\C (\R^d)\times A,$ and $y\in \R^d$:
		\begin{align*}
		|\sigma_x(t,x,\mu,\alpha)|&\leq A_2,\\
		|\sigma_\alpha(t,x,\mu,\alpha)|&\leq A_2,\\
		|\sigma_\mu(t,x,\mu,\alpha)(y)|&\leq A_2(1+|x|).
		\end{align*}
		\item\label{A3} For all $R>0$ and every $(s,x,\mu ,\alpha )\in [0,T]\times \R ^d\times \mathcal P_2 (\R^d)\times A,$ such that 
		$|x|\vee\mathcal{M}_2(\mu)\vee|\alpha|\leq R$
		the quantities
%		\begin{align*}
		\[
		f_x(t,x,\mu,\alpha),\,
%		&\leq A_3(1+R),\\
		f_\alpha(t,x,\mu,\alpha),\,
%		&\leq A_3(1+R),\\
		g_x(x,\mu),\,
%		&\leq A_3(1+R),\\
		\int_{\R ^d}|f_\mu(t,x,\mu,\alpha)(y)|^2\mu(dy),\,
%		&\leq A_3(1+R),\\
		\int_{\R ^d}^{}|g_\mu(x,\mu)(y)|^2\mu(dy),\,
		\]
%		\end{align*}
are all bounded in norm by $A_3(1+R)$.
	\end{enumerate}
\end{assumption}

\begin{example}
	\label{exa:A}
	Again, we investigate the above properties for the setting of a FitzHugh-Nagumo neural network. The property \ref{A3} depends on the choice of $f$ and $g$, hence we do not discuss it here (it is however clear for the ansatz \eqref{choice:cost} below).
	Concerning assumption \ref{A1} and \ref{A2} we have 
	\begin{align*}
	b_x(t,x,\mu,\alpha)=\begin{pmatrix}
	1-v^2-J\beta(\mu) & 1 & 0\\
	c & -cb & 0\\
	\overline{a}S'(v)(1-y) & 0 & -\overline{a}S(v)-\overline{b}
	\end{pmatrix}
	\end{align*}
	where we recall the notation \eqref{barycenter}.
	Using that $\mathrm{Supp}(\mu)\subset\C$, together with the boundedness of $S'(v)$, this leads to
	\begin{align*}
	\langle b_x(t,x,\mu,\alpha)\cdot z,z\rangle \leq A_1(b,c,\overline{a},\overline{b},|S|_\infty,|S'|_{\infty})|z|^2,
	\end{align*}
	hence the first estimate.
	Letting as before $\beta(\mu):=\int_{\R^3}z_3\mu(dz),$
	 it is easily seen by definition of the L-derivative that 
	 \[ 
	 \beta_\mu(\mu)(\tilde x)\cdot h= h_3\quad \text{for all}\enskip \tilde x\enskip\text{and}\enskip h\equiv(h_1,h_2,h_3)\in\R^3 .
	 \] 
	 In a matrix representation, this gives the following constant value for the L-derivative of the drift coefficient at a given point $x\equiv(v,w,y)\in\R^3$ 
	\begin{align*}
	b_\mu(t,x,\mu,\alpha)(\tilde x)=\begin{pmatrix}
	0 & 0 & -J(v-V_{rev})\\
	0 & 0 & 0\\
	0 & 0 & 0
	\end{pmatrix},\quad \text{for all}\enskip\tilde x\in\R^3\,.
	\end{align*}
	Thus we have
	$|b_\mu(t,x,\mu,\alpha)(\tilde x)|\leq J\vee (JV_{rev})(1+|x|),$ showing the desired property.
\end{example}

Next, we introduce the corresponding adjoint equation, which will be essential for the maximum principle. For a solution $X\in \mathcal{S}^{2,d}$ of \eqref{eqa} consider the following backward SDE %with unknown $(P,Q)$:
\begin{equation}
\label{eqb}
\left \{\begin{aligned}
dP_t&=-\Big\{\langle b_x(t,X_t,\LL(X_t),\alpha_t),P_t\rangle 
+\langle \sigma_x(t,X_t,\LL(X_t),\alpha_t),Q_t\rangle 
+f_x(t,X_t,\LL(X_t),\alpha_t)\\
&\quad -\tilde{\e}\left(\langle b_\mu(t,X_t,\LL(X_t),\alpha_t)(\tilde X_t),\tilde{P}_t\rangle +f_\mu(t,X_t,\LL(X_t),\alpha_t)(\tilde X_t)\right)\Big\}dt
+Q_tdW_t\\
P_T&=g_x(X_T,\LL(X_T))+\tilde{\e}\left(g_\mu(X_t,\LL(X_T))(\tilde{X}_T)\right),
\end{aligned}\right .
\end{equation}
where the tilde variables $\tilde X,\tilde P$ are independent copies of the corresponding random variables (carried on some arbitrary probability space $(\tilde\Omega,\mathcal{\tilde F},\mathbb{\tilde P})$), and $\mathbb{\tilde E}$ denotes integration in $\tilde\Omega$
(this convention will be adopted throughout the paper).
Herein, we recall that $\langle\sigma(t,x,\mu,\alpha),q\rangle$ is a synonym for 
$\text{tr}(\sigma(t,x,\mu,\alpha)^\dagger q)$.

A pair of processes $(P,Q)\in \mathcal{H}^{2,d}\times \mathcal{H}^{2,d\times m}$ will be called a solution to the adjoint equation corresponding to $X$ if it satisfies \eqref{eqb} for all $t\in [0,T]$, $\p$-almost surely.\smallskip

We are now in position to formulate the maximum principle. For that purpose, we introduce the Hamiltonian, which for each $x,y,p\in\R^d,$ $q\in\R^{d\times m}$ $\mu\in \mathcal P_2$ and $\alpha\in A,$ is the quantity
\begin{align*}
H(t,x,\mu,p,q,\alpha):=\langle b(t,x,\mu,\alpha),p\rangle 
+\langle \sigma(t,x,\mu,\alpha), q\rangle
+f(t,x,\mu,\alpha)\,.
\end{align*}

\begin{theorem}
	\label{thm:max}
	Let assumptions \ref{H_pi}--\ref{ass:pontryagin} hold. Let $\overline{\alpha}\in \A $ be an optimal control for the problem \eqref{SM}. If $(P,Q)\in \mathcal{H}^{2,d}\times \mathcal{H}^{2,d\times m}$ is the solution to the corresponding adjoint equation, then we have for Lebesgue-almost every $t\in [0,T]$
		\begin{align*}
	&\e\left(H(t,X_t,\LL(X_t),P_t,Q_t,\overline{\alpha}_t)\right)\leq \e\left(H(t,X_t,\LL(X_t),P_t,Q_t,\alpha)\right),
	\end{align*}
	for all $\alpha\in A$.
\end{theorem}

It should be noticed that in contrast to the maximum principle stated in \cite[Thm.~6.14 p.\ 548]{carmona2018probabilistic}, the maximum principle here is formulated in terms of the expectation for almost every $t\in [0,T]$ instead of $dt\otimes \p-$ almost everywhere, since we only consider deterministic controls and thus we only alter the control in deterministic directions.

\section{Well-Posedness of the Optimal Control Problem}
\label{sec:well}

The main purpose of this section is to prove the existence of an optimal control for the stated control problem. For that purpose, we will need to show (among other results) that the state equation \eqref{state} is well-posed, and that the solution satisfies uniform moment bounds up to a certain level. Hereafter, we suppose that assumptions \ref{H_pi}, \ref{ass:mkv} and \ref{ass:weak} are fulfilled.

\subsection{Well-posedness of the State equation}

Our first task is to show that the level-set constraint which was alluded to in Section \ref{sec:level} is preserved along the flow of solutions. This statement is contained the next result. The proof is partially adapted from that of \cite[Prop.~3.3]{bossy2015clarification}.

\begin{lemma}\label{lem:constraint}
	For every $\alpha\in \A $ and $\mu\in C([0,T];\mathcal{P}^{\C}_2(\R^d))$ we have that 
	\begin{align}
	\label{constraint_pi}
	\p\big(\pi(X_t^{\alpha,\mu}) \leq 0,\forall t\in [0,T]\big)=1
	\end{align}
	where $X_t^{\alpha,\mu}$ is the unique solution to 
	\begin{equation}
		\label{fixpointeq}
	\left \{\begin{aligned}
		dX_t
		&=b(t,X_t,\mu_t,\alpha_t)dt+\sigma(t,X_t,\mu_t,\alpha_t)dW_t, 
		& t\in [0,T]
		\\
		X_0
		&\in L^r(\Omega,\mathcal{F}_0,\p;\R^d).
	\end{aligned}\right .
	\end{equation}
\end{lemma}

\begin{proof}
	First, observe that given $\mu\in C([0,T];\mathcal{P}_2(\R^d)),$ equation \eqref{fixpointeq} has a unique strong solution $X^\mu$ in $\mathcal{S}^2.$ Indeed, if we let
	\[ b^\mu(t,x,\alpha):=b(t,x,\mu,\alpha),\quad 
	\sigma^\mu(t,x,\alpha):=\sigma(t,x,\mu,\alpha),
	 \]
 then  from Assumption \ref{L1} we see that $\sigma^\mu$ is Lipshitz, while \ref{L2} and \ref{L3} imply the local Lipschitz continuity and the monotonicity of the drift coefficient $b^\mu$. Hence, by standard results on monotone SDEs (see for instance \cite[Thm.~3.26 p.\ 178]{pardoux2014stochastic}) \eqref{fixpointeq} has a unique strong solution, this solution being progressivey measurable and square integrable. This proves our assertion.\smallskip
	
	In order to show \eqref{constraint_pi}, consider a family $(\Psi_\epsilon)_{\epsilon>0}$ of non-negative and non-decreasing functions in $ C^2(\R)$ which for all $\epsilon>0$ satisfy:
	\[
	\Psi_\epsilon(x)=0\text{ on }(-\infty,0]\,,\quad 
	\Psi_\epsilon(x)=1\text{ on }[\epsilon,\infty)\,,\quad
	\quad\sup_\epsilon|\Psi_\epsilon|_\infty\leq 1\,,
	\]
	and such that $\Psi_\epsilon$ converges pointwise to $\mathbf{1}_{(0,\infty)}$ as $\epsilon\to 0$.
	Let $\tau_n:=\inf\{t\geq 0\text{ s.t.\ }|X_t|\geq n\}$. By It\^o Formula, we have for each $n\geq 0$ and $\epsilon>0$
	\begin{align*}
	\Psi_\epsilon(\pi(X_{t\wedge\tau_n}))-M_t^{\epsilon}
	&=\int_{0}^{\tau_n\wedge t} \big(\pi_x(X_s)\cdot b(s,X_s,\mu_s,\alpha_s)\big)\Psi_\epsilon'(\pi(X_t))ds
	\\
	&\quad \quad 
	+\dfrac{1}{2} \int_{0}^{\tau_n\wedge t}
	\Psi_\epsilon''(\pi(X_s)) |\pi_x(X_s)^\dagger\sigma(s,X_s,\mu_s,\alpha_s)|^2ds
	\\
	&\quad \quad 
	+\frac12\int_{0}^{\tau_n\wedge t}\Psi_\epsilon'(\pi(X_s)) \pi_{xx}(X_s)\cdot(\sigma\sigma^\dagger(s,X_s,\mu_s,\alpha_s)) ds\,,
	\end{align*}
	where we let $M_t^{\epsilon}:=\sum\nolimits_{k=1}^m\int_{0}^{\tau_n\wedge t}
	\pi_x(X_s)\cdot \sigma^{\cdot,k}(s,X_s,\mu_s,\alpha_s) \Psi_\epsilon'(\pi(X_s))dW^k_s$. Since $\Psi_\epsilon$ is supported on the real positive axis, only the values of $X$ which satisfy $\pi(X)>0$ contribute to the above expression. Hence, making use of Assumption \ref{H_pi}, we see that the first term in the previous right hand side is bounded above by $0$, while the two last terms simply vanish.
	We arrive at the relation
	\begin{align*}
	\e\left (\sup_{t\in[0,T]}\Psi_\epsilon(\pi(X_{t\wedge\tau_n}))\right )\leq 0.
	\end{align*}
	Letting first $n\to \infty$, and then $\epsilon\to0,$ we observe by Fatou Lemma that
	\begin{align*}
	\e\left(\sup_{t\in[0,T]}\mathbf{1}_{(0,\infty)}(\pi(X_t))\right)=0,
	\end{align*}
	and our claim follows.
\end{proof}

We are now able to prove the existence of a unique solution to equation \eqref{state}.

\begin{theorem}
	\label{thm:state}
	There exists a unique strong solution to equation \eqref{state} in $\mathcal{S}^2$, which is supported in $\C$ for all times. Furthermore, for each $p\in [2,r]$ and every $\alpha\in\A,$ the solution satisfies the moment estimate
\begin{equation}
\label{moment_est}
	\begin{aligned} 
		\e\left(\sup_{t\in [0,T]}|X_t|^p\right)
		&\leq C\left (\|X_0\|_{p},L_1,L_3,p\right )\left (1+\int_{0}^{T}|\alpha_t|^pdt\right ).
	\end{aligned}
\end{equation}
where the constant $C$ depends only upon the indicated quantities.
\end{theorem}

\begin{proof}
	Recall that $\mathcal{P}_2^\C $ denotes the set of probability measures in $\mathcal{P}_2(\R^d)$ which are supported in
	$\C:=\pi^{-1}((\infty,0]).$
	Equipped with the standard Wasserstein distance, it is a closed subset of $\mathcal{P}_2(\R^d)$. Indeed, it is standard (see for instance \cite{dudley2018real}) that given probability measures $\{\mu_n,n\in\mathbb N\}$ and $\mu$ such that $\mu_n\Rightarrow\mu$, then
	\begin{align*}
	\text{supp}\mu\subset \liminf_{n\rightarrow \infty}(\text{supp}\mu_n):=\left\lbrace  x\in \R^d\enskip\Big|\enskip \limsup_{n\rightarrow \infty}\inf_{y\in \text{supp}\mu_n}|x-y|=0\right\rbrace,
	\end{align*} 
	so that our claim follows.
		Thus, for fixed $\alpha\in\A$, we can rightfully consider the operator 
	\begin{align*}
	\Theta\colon C([0,T];\mathcal{P}_2^\C )\rightarrow C([0,T];\mathcal{P}_2^\C ),\quad\mu\mapsto (\LL(X^{\alpha,\mu}_t))_{t\in [0,T]},
	\end{align*}
	where $X^{\mu}=X^{\alpha,\mu}$ is the unique solution to eq \eqref{fixpointeq}. Using similar arguments as in \cite{carmona2018probabilistic}, the existence of a unique solution to \eqref{eqa} follows if one can show that $\Theta$ has a unique fixed point. In fact, we are going to show that it is a contraction (for a well-chosen metric). The moment estimate \eqref{moment_est} will follow from the fixed point argument, provided one can show that
	\begin{equation}\label{moment_mu}
	\begin{aligned}
	\e\left(\sup_{t\in [0,T]}|X_t^\mu|^p\right)
	&\leq C\left (\|X_0\|_{p},L_1,L_3,p\right )\left (1+\int_{0}^{T}|\alpha_t|^pdt\right )
	\end{aligned}
	\end{equation}
	where the displayed constant depends on the indicated quantities but not on the particular element $\mu$ in $C([0,T];\mathcal P^{\C}_2)$.
	We now divide the proof into two steps.

	\item[\indent\textit{Step 1: moment bounds.}]
%	We content ourselves to show \eqref{moment_mu} for $p=2$ since the general case is merely a matter of notation.
%	 But in that case, 
	 It\^o Formula gives
	\begin{equation}\label{ito_square}
	 \begin{aligned}
	\frac1p|X_t^\mu|^p -N_t
	&=\frac1p|X_0^\mu|^p + \int_0^t\Big\{\left \langle X^\mu, b(s,X^\mu,\mu,\alpha)\right \rangle |X|^{p-2} 
	\\
	&\quad \quad 
	+\frac12|\sigma(s,X^\mu,\mu,\alpha)|^2|X|^{p-2} + \frac{p-2}{2}|\sigma^{\dagger}X^\mu|^2|X|^{p-4} \Big\}ds 
	\end{aligned} 
	\end{equation}
where $N_t:=\int_0^t\left |X|^{p-2}\langle X^\mu, \sigma(s,X^\mu,\mu,\alpha)dW_s\right \rangle$ is the corresponding martingale term. 
Denoting by $\kappa>0$ the constant in the Burkholder-Davis-Gundy Inequality, the latter is estimated thanks to \eqref{hyp:sigma1} and Cauchy-Schwarz Inequality as
\[\begin{aligned}
 \e(\sup_{s\in[0,t]}N_t)
 &\leq \kappa\e\left (\left (\int_0^t|X|^{2p-4}|\sigma(s,X^\mu,\mu,\alpha)^\dagger X^\mu|^2ds\right )^{\frac12}\right )
 \\
 &\leq \kappa \sqrt{L_1}\e\left (\sup_{0\leq s\leq t}|X^\mu|^{\frac p2}\left (\int_0^t|X|^{p-2}(1+|X^\mu|^2+|\alpha|^2)ds\right )^{\frac12}\right ).
\end{aligned}
 \]
But from Young's inequality, the previous right hand side is also bounded by
\[
\frac{1}{2p}\e\left (\sup_{0\leq s\leq t}|X^\mu|^p\right ) + \frac{p\kappa ^2L_1}{2}\e\left (\int_0^t(|X|^{p-2}+|X^\mu|^p+|X|^{p-2}|\alpha|^2)ds\right ).
\]

Define  $\Psi_t:=\e\left (\sup_{0\leq s\leq t}|X_s^\mu|^p\right ).$
Taking the expectation in \eqref{ito_square}, we infer from \eqref{hyp:mono1}, \eqref{hyp:sigma1}, Young's inequality $ab\leq  \frac2pa^{\frac p2}+\frac{p-2}{p}b^{\frac{p}{p-2}}$ and the previous discussion that
\[ 	\frac{1}{2p}\Psi_t
\leq \frac1p\e(|X_0|^p) + C_p\left (L_1+ L_3\right )\int_0^t(1+\Psi_s+|\alpha_s|^p)ds 
 \]	
for some universal constant $C_p>0.$
Applying Gronwall Inequality, we obtain the desired moment estimate.

	% (since the control variable plays no role in the sequel, we will omit to indicate it in the superscript of the state space).
\item[\indent \textit{Step 2: the fixed point argument}.]
 From Lemma \ref{lem:constraint}, it is clear that for all $t\in [0,T],$ the probability measure $\p\circ (X_t^\mu)^{-1}$ is supported in $\C$. 
	For simplicity, let $L:=L_1\vee L_2\vee L_3$ and introduce the weight
	 \[ \phi_t:=\exp\left (-2Lt\right ) ,\quad t\in[0,T].\]
	Then, It\^o Formula gives
\begin{equation}\label{mart_decomp}
\begin{aligned}
	d\left (\frac12|X_t^{\mu}-X_t^{\nu}|^2\phi_t\right )
	&+2L|X_t^{\mu}-X_t^{\nu}|^2\phi_tdt
	\\
&=\phi_t\left \langle  X^{\mu}_t-X^{\nu}_t,b(t,X^{\mu}_t,\mu_t,\alpha_t)-b(t,X^{\nu}_t,\mu_t,\alpha_t)\right \rangle dt
\\
&\quad +\phi_t\left \langle  X^{\mu}_t-X^{\nu}_t,b(t,X^{\nu}_t,\mu_t,\alpha_t)-b(t,X^{\nu}_t,\nu_t,\alpha_t)\right \rangle dt
\\
&\quad \quad + \phi_t\left \langle X^{\mu}_t-X^{\nu}_t,\sigma(t,X^{\mu}_t,\mu_t,\alpha_t)-\sigma(t,X^{\nu}_t,\nu_t,\alpha_t)dW_t\right \rangle 
\\
&\quad \quad \quad + \frac12\phi_t |\sigma(t,X^{\mu}_t,\mu_t,\alpha_t)-\sigma(t,X^{\nu}_t,\nu_t,\alpha_t)|^2 dt.
\end{aligned}
\end{equation}
The first term in the right hand side of \eqref{mart_decomp} is evaluated thanks to \eqref{hyp:mono2}. For the second term, we use the quadratic growth assumption \eqref{hyp:wass_b}.
As for the It\^o correction, we can estimate it similarly, using this time Assumption \ref{L1}.
With $M_t:=\int_{0}^{t} \phi_s\left \langle X^{\mu}_s-X^{\nu}_s,\sigma(s,X^{\mu}_s,\mu_s,\alpha_s)-\sigma(s,X^{\nu}_s,\nu_s,\alpha_s)dW_s\right \rangle$ we get
\[ \begin{aligned}
\frac12|X_t^{\mu}-X_t^{\nu}|^2\phi_t
&+2L \int_0^t|X_s^{\mu}-X_s^{\nu}|^2\phi_sds -M_t
\\
&\leq
\int_0^t\Big\{(L_1+L_3)|X^\mu_s-X^\nu_s|^2
+(L_1+L_2)(1+|X_s^\nu|^2) W_2(\mu_s,\nu_s)^2
\Big\}\phi_sds
\\
&\leq 
2L\left(\int_{0}^{t}|X^\mu_s-X^\nu_s|^2
\phi_sds+
\Big (1+\sup_{s\in [0,t]}|X^\nu_s|^2\Big )\int_0^t W_2(\mu_s,\nu_s)^2\phi_sds
\right) 
\,.
\end{aligned}
 \]
Taking expectations, supremum in $t$, then absorbing to the left yields 
%\[ \begin{aligned}
%&\frac12\sup_{0\leq s\leq t}\e\left(|X_s^{\mu}-X_s^{\nu}|^2\right)\phi_s
%+(\lambda-1) L
%\e\left(\int_0^t|X_s^{\mu}-X_s^{\nu}|^2\phi_sds\right)
%\\
%&\leq 
%L\left (1+\e\left(\sup_{s\in [0,t]}|X^\nu_s|^2\right)\right )\int_0^t W_2(\mu_s,\nu_s)^2\phi_sds
%\\
%&\quad 
%+ \kappa \e\left(\sup_{s\in [0,t]}\sqrt{\phi_s}|X_s^\mu-X_s^\nu|\left(\int_{0}^{t}\phi_s|\sigma(t,X^{\mu}_t,\mu_t,\alpha_t)-\sigma(t,X^{\nu}_t,\nu_t,\alpha_t)|^2ds\right)^\frac{1}{2}\right).
%\end{aligned}
%\]
%Hence, Young's inequality and Assumption \ref{L1} yield
\[ \begin{aligned}
&\sup_{0\leq s\leq t}\e\left(|X_s^{\mu}-X_s^{\nu}|^2\right)\phi_s
%+(\lambda-1-\kappa^2) L\e\left(\int_0^t|X_s^{\mu}-X_s^{\nu}|^2\phi_sds\right)
\leq 
4L\Big(1+\e\big(\sup_{s\in [0,t]}|X^\nu_s|^2\big)\Big)\int_0^t W_2(\mu_s,\nu_s)^2\phi_sds
\end{aligned}
\]
Using the estimate \eqref{moment_est} with $p=2$, the fact that $\exp(-2TL)\leq \phi\leq 1,$ and the basic inequality \eqref{elementary_W2}, we arrive at
\[ \begin{aligned}
\sup_{0\leq s\leq t}
W_2(\Theta(\mu)_s,\Theta(\nu)_s)
&\leq 
C(\|X_0\|_p,T,L,K)\int_0^t W_2(\mu_s,\nu_s)^2ds\,.
\end{aligned}
\]
The contractivity now follows by considering the $k$-th composition of the map $\Theta$, for some $k>0$ large enough and the result then follows from Banach-fixed point theorem. 
\end{proof}

We now investigate some regularity of the control-to-state operator, which will be needed in the proof of the optimality principle.
\begin{lemma}
	\label{lem:lipshitz}
	For $p\in[2,r],$ the solution map 
	\begin{align*}
	G\colon\A \rightarrow \mathcal{S}^p\cap\mathcal{S}^2,\quad\alpha\mapsto X^{\alpha}
	\end{align*}
	is well-defined and Lipschitz continuous. More precisely, there exists a constant $C(L_1,L_2,L_3,T,K)>0$ (here $K$ is the constant associated to $\A$ through \eqref{adm_cont}), such that for all $\alpha,\beta\in \A $ 
	\begin{align*}
	\e\left(\sup_{t\in [0,T]}|X^{\alpha}_t-X^{\beta}_t|^2\right)\leq C(L_1,L_2,L_3,T,K)\int_{0}^{T}|\alpha_t-\beta_t|^2dt.
	\end{align*}
\end{lemma}

\begin{proof}
	That $G$ is well-defined follows immediately from Theorem \ref{thm:state}. Towards Lipschitz-continuity, the property is shown by similar considerations as in the proof of Theorem \ref{thm:state}.
Indeed, fixing $\alpha,\beta\in \A $ and letting $M$ be the martingale $M_t:=\int_0^t\langle\sigma(t,X^{\alpha},\LL(X^{\alpha}),\alpha)-\sigma(s,X^{\beta},\LL(X^{\beta}),\beta),(X^{\alpha}-X^{\beta}) dW \rangle$,
then using It\^o Formula with assumptions \ref{L1}, \ref{L2} and \ref{L3}, we arrive at
	\begin{align*}
	&\frac12|X^{\alpha}_t-X^{\beta}_t|^2-M_t
	\\
	&=
	\int_0^t\Bigg\{\left \langle X^{\alpha}_s-X^{\beta}_t, b(s,X^{\alpha}_s,\LL(X^{\alpha}_s),\alpha_s)-b(s,X^{\beta}_s,\LL(X^{\alpha}_s),\beta_s)\right \rangle 
	\\
	&\quad 
	+\left \langle X^{\alpha}_s-X^{\beta}_s,b(s,X^{\beta}_s,\LL(X^{\alpha}_s),\beta_s)-b(s,X^{\beta}_s,\LL(X^{\beta}_s),\beta_s)\right \rangle 
	\\
	&\quad \quad \quad 
	+ \frac12|\sigma(t,X^{\alpha}_s,\LL(X^{\alpha}_s),\alpha_s)-\sigma(t,X^{\beta}_s,\LL(X^{\beta}_s),\beta_s)|^2\Bigg\}ds
	\\
	&\leq \int_0^t\Big\{ (L_3+\frac12+L_1)(|X^\alpha-X^\beta|^2 + |\alpha-\beta|^ 2) 
	\\
	&\quad \quad \quad \quad 
	+ (\frac{L_2}{2}+L_1)(1+|X^\alpha|^2+|X^\beta|^2)W_2(\LL(X^\alpha),\LL(X^\beta))^2	\Big\}ds\,.
	\end{align*}
Letting $\kappa>0$ be the constant in the BDG inequality, the estimate \eqref{elementary_W2} and $ab\leq \frac {a^2}{4}+ b^2$ yield
	\begin{multline*}
	\frac14\e\left(\sup_{s\in [0,t]}|X^{\alpha}_t-X^{\beta}_t|^2\right)
	\\
	\leq C_L(3+\kappa^2)\Big( 2 + \e\big(\sup_{s\in[0,T]} |X^\alpha_s|^2+|X_s^\beta|^2\big)\Big)\int_{0}^{t}\left \{\e(\sup_{r\in [0,s]}|X_r^{\alpha}-X_r^{\beta}|^2) + |\alpha_s-\beta_s|^2\right\}ds
	\end{multline*}
	where $C_L:=\frac12\vee L_1\vee\frac{L_2}{2}\vee L_3.$
	The result now follows from the uniform bound \eqref{moment_est}, together with Gronwall Lemma.
\end{proof}
\begin{remark}
	Since we have $W_2(\LL(X_s^{\alpha}),\LL(X_s^{\beta}))\leq \e\left( \sup_{t\in [0,T]}|X_s^{\alpha}-X_s^{\beta}|^2\right)^\frac{1}{2}$ we also get the Lipschitz continuity of the map 
	\begin{align*}
	\A \rightarrow \mathcal{P}_2(\mathcal{S}^2),
	\quad\alpha \mapsto \LL(G(\alpha)).
	\end{align*}
\end{remark}

\begin{remark}[Fokker-Planck equation]
	\label{rem:FP}
	Given the settings of Example \ref{exa:FHN},
	we define
	\begin{align*}
	b_0(t,x,\alpha)&:=\begin{pmatrix}
	v-\frac{v^3}{3}-w+\alpha \\ c(v+a-bw) \\ \overline{a}S(v)(1-y)-\overline{b}y
	\end{pmatrix}
	,\quad
	b_1(x,z):=\begin{pmatrix}
	-J(v-V_{rev})z_3 \\ 0 \\ 0
	\end{pmatrix},
	\\
	\tilde{\sigma}(x,z)&:=\begin{pmatrix}
	\sigma_{ext} & -\sigma^J(v-V_{rev})z_3 & 0 \\
	0 & 0 & 0 \\
	0 & 0 & \chi(y)\sqrt{\overline{a}S(v)(1-y)+\overline{b}y} 
	\end{pmatrix}.
	\end{align*}
	If we assume that the solution to the corresponding mean-field equation has a density $p(t,x)$ with respect to the $3$-dimensional lebesgue measure, then the McKean-Vlasov-Fokker-Planck equation is given by the nonlinear PDE:
	\begin{align*}
	\partial_tp(t,x)&=-\text{div}\left(\left(b_0(t,x,\alpha)+\int_{\R^3}b_1(x,z)p(t,z)dz\right) p(t,x) \right)\\
	&\quad +\dfrac{1}{2}\nabla^2\cdot\left(\left(\iint_{\R^3\times \R^3}\tilde{\sigma}(x,z)\tilde{\sigma}(x,\bar z)^{\dagger}p(t,z)p(t,\bar z)\,dz d\bar z\right)p(t,x)\right)
	\end{align*}
	(see \cite{baladron2012mean}). It is degenerate parabolic because the matrix $\sigma {\tilde \sigma}^\dagger$ is not strictly positive.
\end{remark}

\subsection{Proof of Theorem \ref{thm:ex}}

We now prove the existence of an optimal control for \eqref{eqa}. The strategy we use strings along the commonly named ``direct method'' in the calculus of variations. As a trivial consequence of the assumptions made in Section \ref{sec:cost} and the uniform estimate \eqref{moment_est}, note at first that our control problem is indeed finite.
 Next, consider a sequence $(\alpha^n)_{n\in \mathbb{N}}\subset \A $ realizing the infimum of $J$ asymptotically, i.e.\
	\begin{align*}
	\lim\limits_{n\rightarrow \infty}J(\alpha^n)=\inf_{\alpha\in \A }J(\alpha).
	\end{align*} 
	Since $\A \subset L^2([0,T];\R^k)$ is bounded and closed, by Banach Alaogu Theorem
	there exists  an $\alpha\in L^2([0,T];\R^k)$ and a subsequence also denoted by $(\alpha^n)_{n\in \mathbb{N}}$, such that 
	\begin{align*}
	\alpha^n\rightharpoonup \alpha, \quad \text{weakly in}\enskip L^2(0,T;\R^k).
	\end{align*}
	Since $\A $ is also convex, we get $\alpha\in \A $, so $\alpha$ is indeed an admissible control. We now divide the proof into four steps.
	
	\begin{trivlist}
	\item[\indent\textit{Step 1: tightness.}]
	In the sequel, we denote by $X^n$ the solution of the state equation \eqref{state} with respect to the control $\alpha^n$, $n\in\mathbb N.$ Adding and subtracting in \eqref{state}, we have
	\begin{multline*}
	\|X_t^n-X_s^n\|_4^4
	\leq 4^3\bigg\{
	\big\|\int_{s}^{t}b(r,0,\delta_0,0)dr\big\|_4^4
	+ \big\|\int_s^tb(r,0,\LL(X_r^n),0)-b(r,0,\delta_0,0)dr\big\|_4^4
	\\
	+\big\|\int_{s}^{t}b(r,X_r^n,\LL(X_r^n),\alpha_r^n)-b(r,0,\LL(X_r^n),0)dr\big\|_4^4
	+\kappa\big\|\int_{s}^{t}|\sigma(r,X^n_r,\LL(X^n_r),\alpha^n_r)|^2dr\big\|_2^2\bigg\}\,,
	\end{multline*}
	where $\kappa>0$ is the constant in the BDG inequality.
	Using the assumptions \ref{L1}, \ref{L2}, \ref{L3}, the fact that $0\in\C$ and the basic inequality \eqref{elementary_W2}, we obtain that 
%	(hereafter, we let $|\cdot|_{C_T}$ be the supremum norm on $[0,T]$):
	\begin{multline*}
	\|X_t^n-X_s^n\|_4^4
	\leq 4^3\bigg\{
	(t-s)^4 \sup_{r\in[0,T]}|b(r,0,\delta_0,0)|^4
	+ (t-s)^4L_2^2\sup_{r\in[0,T]}|X^n_r|^4
	\\
	+
	L_2^2\big\|\int_{s}^{t}(1+|X^n_r|^{q-1}+|\alpha^n_r|^{q-1}+\mathcal{M}_2(\LL(X_r^n))^2(|X^n_r|+|\alpha^n_r|)dr\big\|_4^4
	\\
	+  \kappa L_1^2\big\|\int_{s}^{t}(1+|X_r^n|^2+|\alpha_r^n|^2)dr\|_2^2
	\bigg\}\,.
	\end{multline*}
	Using H\"older Inequality, our assumption that $4\leq 4q\leq r$ together with Young Inequality $ab\leq \frac{q-1}{q}a^\frac{q}{q-1}+\frac{1}{q}b^q$, we arrive at the following estimate, for all $n\in \mathbb{N}$ and $0\leq s\leq t\leq T$
	\begin{align*}
	\e\left(|X_t^n-X_s^n|^4\right)
	&\leq C(L,T)\bigg\{ 
	(t-s)^4 \left [\sup_{r\in[0,T]}|b(r,0,\delta_0,0)|^4
	+\e\left(1+\sup_{r\in[0,T]}|X^n_r|^{4q}\right)\right]
	\\
	&\quad 
	+ (t-s)^{4/3}\Big(1+\int_s^t|\alpha_u^n|^rdu\Big)\bigg\}\,,
	\end{align*}
	where the above constant depends upon the indicated quantities, but not on $n\in\mathbb N.$

	Making use of the uniform estimate \eqref{moment_est}, the Kolmogorov continuity criterion then asserts that the sequence of probability measures $(\p\circ (X^n)^{-1})_{n\in \mathbb{N}}$, defined on the space 
	\[E:=\left (C([0,T];\R^d),\mathcal{B}(C([0,T];\R^d))\right )\] is tight.
	In the same way, we can prove that the sequence on probability measures $(\p_n)_{n\in \mathbb{N}}:=(\p\circ (X^n)^{-1},\p\circ (B^n)^{-1})_{n\in \mathbb{N}}$, with 
	\begin{align*}
	B^n(t):=\int_{0}^{t}b(s,X_s^n,\LL(X_s^n),\alpha_s^n)ds,
	\end{align*}
	is tight on the product space $E\times E,$ with respect to the product topology.
	Thus by Prokhorov's theorem there exists a subsequence of $(\p_n)_{n\in \mathbb{N}}$, which converges weakly to some probability measure $\p^*$ on $E\times E$.

	\item[\indent\textit{Step 2: passage to the limit in the drift.}]
	By Skorokhod's representation theorem we can then find random variables $\overline{X},\overline{B}$, $(\overline{X}^n)_{n\in \mathbb{N}},(\overline{B}^n)_{n\in \mathbb{N}}$ defined on some probability space $(\overline{\Omega},\overline{\mathcal{F}},\overline{\p})$ and with values in $E\times E$ such that 
	\begin{itemize}
		\item $\overline{\p}\circ (\overline{X}^n,\overline{B}^n)^{-1}=\p_n$ for all $n\in \mathbb{N}$ and $\overline{\p}\circ (\overline{X},\overline{B})^{-1}=\p^*$ and
		\item  $\lim\limits_{n\rightarrow \infty}(\overline{X}^n,\overline{B}^n) =(\overline{X},\overline{B})$,  $\overline{\p}$-almost surely with respect to the uniform topology.
	\end{itemize}
 From (\ref{moment_mu}) and by the definition of $\mathbb{A}$ we get for any $p\leq r$
 \begin{align*}
 \overline\e\left(\sup_{0\leq t\leq T}|\overline{X}_t^n|^p\right)\leq C(p,\|X_0\|_p,L_1,L_2,L_3,K),
 \end{align*}
 for some constant independent of $n$. Thus we can conclude by the dominated convergence theorem that
 \begin{align*}
 W_2(\mathcal{L}(\overline{X}_t^n),\mathcal{L}(\overline{X}_t))^2\leq \e\left(\sup_{0\leq t\leq T}|\overline{X}_t^n-\overline{X}_t|^2\right)\rightarrow 0,
 \end{align*}
 as $n\rightarrow \infty$. This also implies $(\mathcal{L}(\overline{X}_t))_{t\in [0,T]}\subset \mathcal{P}_2^\C $, since $ \mathcal{P}_2^\C $ is closed.

 To identify the almost sure limit $\overline{B}$, we first claim that for each $t\in [0,T]$
 \begin{align}
 \label{weak_B}
 \overline{B}^n(t)\rightharpoonup \int_{0}^{t}b(s,\overline{X}_s,\LL(\overline{X}_s),\alpha_s)ds,
 \end{align}
weakly in $L^2(\overline{\Omega};\R^d)$. 
Indeed, by \eqref{hyp:loc_lip} and the dominated convergence theorem we have
 \begin{align*}
 \overline\e\left(\int_{0}^{t}|b(s,\overline{X}_s^n,\mathcal{L}(\overline{X}_s^n),\alpha_s^n)-b(s,\overline{X}_s,\mathcal{L}(\overline{X}_s),\alpha_s^n)|^2ds\right)\rightarrow 0.
 \end{align*}
 Likewise, for $h\in L^2(\overline{\Omega};\R^d)$ we have by Assumption \ref{ass:weak} and dominated convergence
\begin{align*}
\overline\e\left(\int_{0}^{t}\langle\left(b(s,\overline{X}_s,\mathcal{L}(\overline{X}_s),\alpha_s^n)-b(s,\overline{X}_s,\mathcal{L}(\overline{X}_s),\alpha_s)\right),h\rangle ds\right)\rightarrow 0,
\end{align*}
as $n\rightarrow \infty$, thus proving our claim. 

The desired identification then follows from \eqref{weak_B}, the Banach-Saks theorem and the uniqueness of the almost sure limit. 
The processes $\overline B$ and $\int_{0}^{\cdot}b(s,\overline{X}_s,\LL(\overline{X}_s),\alpha_s)ds$ being both continuous pathwise, they are indistinguishable, hence the identity
\begin{align}
\label{id:B_bar}
\overline{B}(t)=\int_{0}^{t}b(s,\overline{X}_s,\LL(\overline{X}_s),\alpha_s)ds,\quad 
\end{align}
for all $t\in[0,T],$
$\overline{\p}$-almost surely.

% \begin{align*}
% &\overline\e\left(\int_{0}^{t}\langle\left(b(s,\overline{X}_s^n,\mathcal{L}(\overline{X}_s^n),\alpha_s^n)-b(s,\overline{X}_s,\mathcal{L}(\overline{X}_s),\alpha_s)\right),h\rangle ds\right)\\
% &\leq C_T \Bigg\{ \overline\e\left(\int_{0}^{t}|b(s,\overline{X}_s^n,\mathcal{L}(\overline{X}_s^n),\alpha_s^n)-b(s,\overline{X}_s,\mathcal{L}(\overline{X}_s),\alpha_s^n)|^2ds\right)^\frac{1}{2}\|h\|_{L^2(\overline\Omega,\R^d)}\\
% &\quad +\overline\e\left(\int_{0}^{t}\langle\left(b(s,\overline{X}_s,\mathcal{L}(\overline{X}_s),\alpha_s^n)-b(s,\overline{X}_s,\mathcal{L}(\overline{X}_s),\alpha_s)\right),h\rangle ds\right)\Bigg\}.
% \end{align*}

	\item[\indent\textit{Step 3: identification of the martingale}.]
	Letting $\sigma\sigma^\dagger(t,x,\mu,\alpha):=\sigma(t,x,\mu,\alpha)\sigma(t,x,\mu,\alpha)^\dagger$
	for short, similar arguments as above show that 
	\begin{align*}
	\sigma\sigma^\dagger(t,\overline{X}^n_t,\LL(\overline{X}^n_t),\alpha^n_t)\rightharpoonup \sigma\sigma^\dagger(t,\overline{X}_t,\LL(\overline{X}_t),\alpha_t)
	\end{align*}
	weakly in $L^2([0,T]\times\overline{\Omega};\R^d)$.
	Since the process 
	\begin{align*}
	M_t^n:=X_t^n-X_0-B^n(t)=\int_{0}^{t}\sigma(s,X_s^n,\LL(X_s^n),\alpha_s^n)dW_s
	\end{align*}
	is, for each $n,$  a $\mathcal{G}_t^n:=\sigma({X_s^n|s\leq t})$ martingale under $\p$, we can conclude that 
	\begin{align*}
	\overline{M}_t^n:=\overline{X}_t^n-X_0-\overline{B}^n(t)
	\end{align*}
	is a $\mathcal{\overline{G}}_t^n:=\sigma({\overline{X}_s^n|s\leq t})$ martingale under $\overline{\p}$ with quadratic variation 
	\begin{align*}
	\langle \overline{M}^n\rangle_t=\int_{0}^{t}\sigma\sigma^\dagger(s,\overline{X}_s^n,\LL(\overline{X}_s^n),\alpha_s^n)ds.
	\end{align*}
	From the previous considerations, we can conclude that 
	\begin{align*}
	\overline{M}_t^n\rightarrow \overline{X}_t-X_0-\int_{0}^{t}b(s,\overline{X}_s,\LL(\overline{X}_s),\overline{\alpha}_s)ds=:\overline{M}_t,
	\end{align*}
	$\overline{\p}$-almost surely for all $t\in [0,T]$. Thus by the dominated convergence theorem the process $(\overline{M}_t)_{t\in [0,T]}$ is a $\overline{\mathcal{G}}_t:=\sigma(\overline{X}_s|s\leq t)$ martingale under $\overline{\p}$ and with standard arguments we also obtain, that $(\overline{M}_t)_{t\in [0,T]}$ has quadratic variation 
	\begin{align*}
	\langle\overline{M}\rangle_t=\int_{0}^{t}\sigma\sigma^\dagger(s,\overline{X}_s,\LL(\overline{X}_s),\overline{\alpha}_s)ds.
	\end{align*}
	By the martingale representation theorem we can find an extended probability space $(\hat{\Omega},\hat{\mathcal{F}},(\hat{\mathcal{F}}_t)_{t\in [0,T]},\hat{\p})$ with an $m$-dimensional brownian motion $\hat{W}$, such that the natural extension $\hat{X}$ of $\overline{X}$ satisfies $\hat{\p}\circ (\hat{X}^{-1})=\overline{\p}\circ(\overline{X}^{-1})$ and 
	\begin{align*}
	\hat{X}_t=X_0+\int_{0}^{t}b(s,\hat{X}_s,\LL(\hat{X}_s),\overline{\alpha}_s)ds+\int_{0}^{t}\sigma(s,\hat{X}_s,\LL(\hat{X}_s),\overline{\alpha}_s)d\hat{W}_s,
	\end{align*}
	$\hat{\p}$-almost surely for all $t\in [0,T]$.

	\item[\indent\textit{Step 4: end of the proof}]
	It remains to show that the infimum is attained for $\alpha .$
	Due to the uniqueness of equation \eqref{state}, we have $\p \circ (X^{\alpha})^{-1}=\hat{\p}\circ (\hat{X}^{-1})$. Using Fatou's lemma, continuity of $f,g,$, Assumption \ref{ass:weak} and Remark \ref{rem:lower}, we obtain 
	\begin{align*}
	\inf_{\alpha\in \A}J(\alpha)&=\lim\limits_{n\rightarrow \infty}J(\alpha^n)\\
	&\geq \liminf_{n\rightarrow \infty}\e\left(\int_{0}^{T}f(t,X_t^n,\LL(X_t^n),\alpha_t^n)dt+g(X_T^n,\LL(X_T^n))\right)\\
	&=\liminf_{n\rightarrow \infty}\overline{\e}\left(\int_{0}^{T}f(t,\overline{X}_t^n,\LL(\overline{X}_t^n),\alpha_t^n)dt+g(\overline{X}_T^n,\LL(\overline{X}_T^n))\right)\\
	&\geq\overline{\e}\left(\int_{0}^{T}f(t,\overline{X}_t,\LL(\overline{X}_t),\alpha_t)dt+g(\overline{X}_T,\LL(\overline{X}_T))\right)\\
	&=\hat{\e}\left(\int_{0}^{T}f(t,\hat{X}_t,\LL(\hat{X}_t),\alpha_t)dt+g(\hat{X}_T,\LL(\hat{X}_T))\right)\\
	&=\e\left(\int_{0}^{T}f(t,X^{\alpha}_t,\LL(X^{\alpha}_t),\alpha_t)dt+g(X^{\alpha}_T,\LL(X^{\alpha}_T))\right)\\
	&=J(\alpha).
	\end{align*} 
	This shows that $\alpha$ has the desired properties, and hence the proof is finished.
\hfill\qed
\end{trivlist}

%{\color{violet}
%\begin{proof}[New end of proof]
% With similar arguments as above we can conclude that 
%	\begin{align*}
%	\sigma(t,\overline{X}^n_t,\LL(\overline{X}^n_t),\alpha^n_t)\sigma(t,\overline{X}^n_t,\LL(\overline{X}^n_t),\alpha^n_t)^\dagger\rightharpoonup \sigma(t,\overline{X}_t,\LL(\overline{X}_t),\alpha_t)\sigma(t,\overline{X}_t,\LL(\overline{X}_t),\overline{\alpha}_t)^\dagger
%	\end{align*}
%	in $L^2([0,T]\times\overline{\Omega},\R^d)$.
%\end{proof}
%}

\section{The maximum principle: proof of Theorem \ref{thm:max}}
\label{sec:max}

In this section, it will be assumed implicitly that assumptions \ref{H_pi}, \ref{ass:mkv}, \ref{ass:weak} and \ref{ass:pontryagin} hold.
Hereafter, we let $(\tilde{\Omega},\tilde{\mathcal{A}},\tilde{\p})$ be a copy of the probability space $(\Omega,\mathcal{A},\p)$. The corresponding expectation map will be denoted by $\tilde\e$. 

\subsection{G\^ateaux differentiability}

In this subsection we aim to complete Lemma \ref{lem:lipshitz} by showing the G\^ateaux-differentiability of the control-to-state operator
\begin{align*}
G:\A\subset L^p([0,T];\R^k)\rightarrow \mathcal{S}^2,\quad\alpha \mapsto X^\alpha.
\end{align*}
 The G\^ateaux derivative of the solution map will be given by the solution of a mean-field equation with random coefficients. We will deal with this problem in the similar fashion as its done in \cite[Thm.~6.10 p.\ 544]{carmona2018probabilistic}.
\begin{lemma}
	\label{lem:variation}
	The solution map $G$ is G\^ateaux-differentiable. Moreover, for each $\alpha\in \A$, its derivative in the direction $\beta\in\A$ is given by 
	\begin{align*}
	dG(\alpha)\cdot \beta=Z^{\alpha,\beta},
	\end{align*}
	where, introducing
	\begin{align*}
	B_\mu(t,x,\mu)&:=\iint_{\R^d\times\R^d}b_\mu(t,x,\mathcal{L}(X_t),\alpha_t)(\tilde{x})\cdot \tilde{y}\mu(d\tilde{x}\times d\tilde{y})
	\\
	\Sigma_\mu(t,x,\mu)&:=\iint_{\R^d\times\R^d}\sigma_\mu(t,x,\mathcal{L}(X_t),\alpha_t)(\tilde{x})\cdot \tilde{y}\mu(d\tilde{x}\times d\tilde{y})\,,
	\end{align*}
	 the process $Z=Z^{\alpha,\beta}$ is characterized as the unique solution to 
		\begin{equation}\label{variation_process}
		\left \{\begin{aligned}
		dZ_t&=\big\{b_x(t,X_t,\LL(X_t),\alpha_t)\cdot Z_t+b_\alpha(t,X_t,\LL(X_t),\alpha_t)\cdot \beta_t+B_\mu(t,X_t,\LL(X_t,Z_t))\big\}dt,\\
		&\quad + \big\{\sigma_x(t,X_t,\LL(X_t),\alpha_t)\cdot Z_t+\sigma_\alpha(t,X_t,\LL(X_t),\alpha_t)\cdot \beta_t+\Sigma_\mu(t,X_t,\LL(X_t,Z_t))\big\}dW_t\\
		Z_0&=0.
		\end{aligned}\right.
		\end{equation}
	
\end{lemma}

\begin{proof}
	We will start by showing that \eqref{variation_process} has a unique solution. For that purpose, we define
	\begin{align*}
	\mathcal{R}:=\left \{\mu\in C([0,T];\mathcal{P}_2(\R^d\times\R^d)),\enskip\text{such that}\enskip\mu_t\circ p_1^{-1}=\LL(X_t),\enskip\forall t\right \},
	\end{align*}
	where $p_1$ denotes the projector onto the first $d$-coordinates, namely
	\begin{align*}
	p_1\colon\R^d\times \R^d\rightarrow \R^d,\quad (x,y)\mapsto x.
	\end{align*}
	Clearly, if $\mu^n_t$ is a sequence converging weakly to $\mu_t$ for every $t\in [0,T]$, the constraint $\mu_t^n\circ p_1^{-1}=\LL(X_t),\forall t$ remains true for $\mu$ itself.
	Since the Wasserstein distance metrizes the weak topology, we see that $\mathcal{R}$ is closed in $C([0,T];\mathcal P_2(\R^d\times \R^d))$. 
	Next, define 
	\begin{align*}
	\Psi\colon \mathcal R\rightarrow \mathcal R,
	\end{align*}
	which maps $\mu \in C([0,T];\mathcal{P}_2(\R^d\times\R^d))$ to $(\LL(X_t,V_t))_{t\in[0,T]}$, where $(V_t)_{t\in [0,T]}$ is the unique solution to 
	\begin{equation}
	\label{eq:V}
	\left \{\begin{aligned}
	dV_t&=b_x(t,X_t,\LL(X_t),\alpha_t)\cdot V_t+b_\alpha(t,X_t,\LL(X_t),\alpha_t)\cdot \beta_t+B_\mu(t,X_t,\mu_t)dt\\
	&\quad + \sigma_x(t,X_t,\LL(X_t),\alpha_t)\cdot V_t+\sigma_\alpha(t,X_t,\LL(X_t),\alpha_t)\cdot \beta_t+\Sigma_\mu(t,X_t,\mu_t)dW_t\\
	Z_0&=0.
	\end{aligned}\right.
	\end{equation}
	For fixed $\mu \in C([0,T];\mathcal{P}_2(\R^d\times\R^d))$ we first need to check the existence of a unique solution $V$. But letting 
	\begin{align*}
	B(t,\omega,v,\mu,\alpha)
	&:=b_x(t,X_t(\omega),\mathcal{L}(X_t),\alpha_t)\cdot v+b_\alpha(t,X_t(\omega),\mathcal{L}(X_t),\alpha_t)\cdot \beta_t+B_\mu(t,X_t(\omega),\mu_t),
	\\
	\Sigma(t,\omega,v,\mu,\alpha)
	&:=\sigma_x(t,X_t(\omega),\mathcal{L}(X_t),\alpha_t)\cdot v+\sigma_\alpha(t,X_t(\omega),\mathcal{L}(X_t),\alpha_t)\cdot \beta_t+\Sigma_\mu(t,X_t(\omega),\mu_t),
	\end{align*} 
	we have the following properties:
	\begin{align*}
	\langle B(t,\omega,v,\mu,\alpha)-B(t,\omega,v',\mu,\alpha),v-v'\rangle&\leq A_1|v-v'|^2
	\\
	\int_{0}^{T}\sup_{|v|\leq c}|B(t,\omega,v,\mu,\alpha)|dt&<\infty,\quad
	\forall c\geq0,
	\end{align*}
	for all $t\in [0,T]$ and $\p$-almost every $\omega$. The first estimate is a result of Assumption \ref{ass:pontryagin} and the fact that $\p(X_t\in\C,\forall t)=1$. The second estimate follows from
	\begin{align*}
	|B(t,\omega,v,\mu,\alpha)|\leq C\Bigg\{(1+|X_t(\omega)|^{q-1})|v|+|\beta_t|+(1+|X_t(\omega)|)\iint_{\R^d\times \R^d}|y|\mu_t(dx\times dy)\Bigg\},
	\end{align*}
	together with the continuity of 
	$t\mapsto \iint_{\R^d\times \R^d}|y|\mu_t(dx\times dy),$
	and the uniform estimate \eqref{moment_est}.
%	\begin{align*}
%	\int_{\Omega}\p(d\omega)\int_{0}^{T}\sup_{|v|\leq c}|B(t,\omega,v,\mu,\alpha)|dt\leq C,
%	\end{align*}
%	for some $C>0$.
	 Using \eqref{constraint_pi} we get with similar arguments 
	\begin{align*}
	|\Sigma(t,\omega,v,\mu,\alpha)-\Sigma(t,\omega,v',\mu,\alpha)|
	&\leq A_2|v-v'|,
	\\
	\int_{0}^{T}\sup_{|v|\leq c}|\Sigma(t,\omega,v,\mu,\alpha)|^2dt
	&<\infty,\quad\forall c\geq0,
	\end{align*}
	for all $t\in [0,T]$, $\p$-almost every $\omega$.
	It follows then by classical SDE results that \eqref{eq:V} is well-posed. Moreover, adapting the arguments yielding the moment estimates of Theorem \ref{thm:state}, it is shown \textit{mutatis mutandis} that  for $2\leq p\leq r$
	\begin{align*}
	\e\left(\sup_{0\leq t\leq T}|V_t|^p\right)<\infty.
	\end{align*}
	Therefore $(V_t)$ (and hence $\Psi(\mu)\equiv\LL(X,V)$) is uniquely determined by the probability measure $\mu$.\smallskip

	 We now aim to prove that $\Psi$ is a contraction, but for that purpose it is convenient to introduce another (stronger) metric. For any $\mu,\nu\in \mathcal P_2(\R^d\times\R^d)$ with $\mu\circ p_1^{-1}=\nu\circ p_1^{-1}$, we let
	 \[
	 d(\mu,\nu)^2:= \inf _{m\in\Lambda(\mu,\nu)}\iiint_{\R^d\times\R^d\times\R^d}|v-w|^2 m(dx\times dv\times dw)\,,
	 \]
	 where $\Lambda(\mu,\nu)$ is the set of all probability measures $m$ on $(\R^d)^3$ such that for any $A,B\in \mathcal B(\R^d)$
	 \[
	 m(A\times B\times \R^d)=\mu (A\times B)\quad\text{and}\quad  
	 m(A\times \R^d\times B)=\nu (A\times B).
	 \]
	 That $d$ is stronger than $W_2$ can be seen as follows. If $m$ is any element in $\Lambda(\mu,\nu)$, one can define 
	 \[
	 \rho(dx\times dv\times dy\times dw):=
	 m(dx\times dv\times dw)\delta_x(dy)
	 \]
	 where $\delta_x$ is the Dirac mass centered at $x$.
	 Clearly, $\rho$ belongs to the set of transport plans $\Pi(\mu,\nu)$ between $\mu$ and $\nu,$ so that in particular 
	 \[
	 W_2(\mu,\nu)=\inf _{\rho\in\Pi(\mu,\nu)}\iiiint\limits_{(\R^d)^4}(|x-y|^2+|v-w|^2 )\pi(dx\times dv\times dy\times dw) \leq \iiint\limits_{(\R^d)^3}|v-w|m(dx\times dv\times dw).
	 \]
	 Then, taking the infimum over all such $m$ yields our conclusion.

Next, let $m\in \Lambda (\mu,\nu)$. Using the marginal condition on $m$, we have 
	\begin{align*}
	&|B_\mu(t,X_t,\mu_t)-B_\mu(t,X_t,\nu_t)|
	\\
	&=
	\Big|\iint_{\R^d\times \R^d}b_\mu(t,X_t,\mathcal{L}(X_t),\alpha_t)(x)\cdot v \mu (dx\times dv)
	\\
	&\quad \quad \quad \quad 
	-\iint_{\R^d\times \R^d}b_\mu(t,X_t,\mathcal{L}(X_t),\alpha_t)\cdot w\nu(dx\times dw)\Big|
	\\
	&=
	\Big|\iiint_{\R^d\times\R^d\times  \R^d}b_\mu(t,X_t,\mathcal{L}(X_t),\alpha_t)(x)\cdot v m (dx\times dv\times dw)
	\\
	&\quad\quad \quad \quad 
	 -\iiint_{\R^d\times \R^d\times \R^d}b_\mu(t,X_t,\mathcal{L}(X_t),\alpha_t)\cdot wm(dx\times dv\times dw)\Big|
	\,.
	\end{align*}
	Thus, \[
	|B_\mu(t,X_t,\mu_t)-B_\mu(t,X_t,\nu_t)|=\Big|\iiint_{\R^d\times\R^d\times  \R^d}b_\mu(t,X_t,\mathcal{L}(X_t),\alpha_t)(x)\cdot (v-w) m(dx\times dv\times dw)\Big|\,.
	\]
	Since $m$ is arbitrary, we obtain
	\begin{align*}
|B_\mu(t,X_t,\mu_t)-B_\mu(t,X_t,\nu_t)|\leq A_1(1+|X_t|)d(\mu_t,\nu_t)\,,
	\end{align*}
	and a similar result can be shown for $\Sigma_\mu$. Now, if we equip $\mathcal R$ with a metric $\delta$ inherited from $d,$ for instance $\delta(\mu,\nu):=\sup_{t\in [0,T]}e^{-\gamma t}d(\mu_t,\nu_t)$ for $\gamma>0$ large enough, the proof that $\Psi$ is a contraction follows with simple arguments. Since it is similar to the proof of Theorem \ref{thm:state}, we omit the details.\smallskip

	Let now $\alpha,\beta\in \A$ and $\epsilon>0$ small enough, such that $\alpha+\epsilon\beta \in \A$. By $X$ we denote the solution of \eqref{state} with respect to $\alpha$ and by $X^\epsilon$ we denote the solution to \eqref{state} with respect to $\alpha+\epsilon\beta$. Furthermore for $\lambda\in [0,1]$ we introduce $X^{\lambda,\epsilon}:=X+\lambda(X^\epsilon-X)$ and $\alpha^{\lambda,\epsilon}:=\alpha+\lambda\epsilon\beta$. 
	Note that, since $\pi$ is convex, we have 
	\begin{align}
	\label{X_lambda_C}
	\pi(X_t+\lambda(X_t^\epsilon-X_t))=\pi((1-\lambda)X_t+\lambda X_t^\epsilon)\leq (1-\lambda)\pi(X_t)+\lambda \pi(X_t^\epsilon)\leq 0\,,
	\end{align}
	hence $X_t^{\lambda,\epsilon}$ is supported in $\C $.
%	 Using the result from Lemma \ref{lem:constraint} we have $\rho(X_t^{\lambda,\epsilon})\leq |\rho|_{\C ,\infty}$.

	Next, by Lemma \ref{lem:lipshitz} we get 
	\begin{align*}
	\e\left(\sup_{\lambda\in[0,1]}\sup_{t\in [0,T]}|X_t^{\lambda,\epsilon}-X_t|^2\right)&\leq \hat{C}_{L,T}\epsilon^2\int_{0}^{T}|\beta_t|^2dt\,.
	\end{align*}
	Thus, we can conclude that $X^{\lambda,\epsilon}\underset{\epsilon\to 0}{\longrightarrow} X$ in $L^2(\Omega,C([0,T];\R^d))$, uniformly in $\lambda$. By a simple Taylor expansion we get
	\begin{multline*}
	b(t,X_t^\epsilon,\LL(X_t^\epsilon),\alpha_t+\epsilon\beta_t)
	\\
	=b(t,X_t,\LL(X_t),\alpha_t)+[b_x]^\epsilon_{t}\cdot (X_t^\epsilon-X_t)
	+\epsilon [b_\alpha]^\epsilon_{t}\cdot \beta_t
	+\tilde{\e}\left([b_\mu]^\epsilon\cdot \widetilde{(X_t^\epsilon-X_t)}\right)
	\end{multline*}
	where, given $\varphi=\varphi(t,x,\mu,\alpha)(\tilde x)$ we use the shorthand notation
	\[	[\varphi]^\epsilon_{t}:=\int_{0}^{1}\varphi\left (t,X_t^{\lambda,\epsilon},\LL(X_t^{\lambda,\epsilon}),\alpha_t^{\lambda,\epsilon}\right )\left (\tilde X^{\lambda,\epsilon}_t\right )d\lambda \,,\]
	with the convention that the last input is ignored whenever $\varphi$ does not depend on the tilde variable.
	Similarly, we have 
	\begin{multline*}
	\sigma(t,X_t^\epsilon,\LL(X_t^\epsilon),\alpha_t+\epsilon\beta_t)
	\\
	=\sigma(t,X_t,\LL(X_t),\alpha_t)+[\sigma_x]^\epsilon_{t}\cdot (X_t^\epsilon-X_t)+\epsilon[\sigma_\alpha]^\epsilon_{t}\cdot \beta_t+\tilde{\e}\left([\sigma_\mu]^\epsilon\cdot \widetilde{(X_t^\epsilon-X_t)}\right).
	\end{multline*}
	Thus, for $\varDelta ^\epsilon_t:=\dfrac{X_t^\epsilon-X_t}{\epsilon}-Z_t^{\alpha,\beta}$ we have
	\begin{align*}
	d\varDelta_t^\epsilon
	&=\Bigg\{[b_x]^\epsilon_{t}\cdot \varDelta_t^\epsilon+\tilde{\e}\left([b_\mu]^\epsilon_{t}\cdot \tilde{\varDelta}_t^\epsilon\right) 
	+ ([b_x]^\epsilon_{t}-b_x(t,X_t,\LL(X_t),\alpha_t))\cdot Z_t^{\alpha,\beta}
	\\
	&\quad\quad \quad 
	 +\epsilon([b_\alpha]^\epsilon_{t}- b_\alpha(t,X_t,\LL(X_t),\alpha_t))\cdot \beta_t\\
	&\quad\quad \quad 
	 +\tilde{\e}\left(([b_\mu]^\epsilon_{t}-b_\mu(t,X_t,\LL(X_t),\alpha_t)(\tilde{X}_t))\cdot \tilde{Z}_t^{\alpha,\beta}\right)\Bigg\}dt\\
	&\quad 
	+ \Bigg\{[\sigma_x]^\epsilon_{t}\cdot \varDelta_t^\epsilon+\tilde{\e}\left([\sigma_\mu]^\epsilon_{t}\cdot \tilde{\varDelta}_t^\epsilon\right) + ([\sigma_x]^\epsilon_{t}-\sigma_x(t,X_t,\LL(X_t),\alpha_t))\cdot Z_t^{\alpha,\beta}\\
	&\quad \quad \quad 
	 +\epsilon([\sigma_\alpha]^\epsilon_{t}- \sigma_\alpha(t,X_t,\LL(X_t),\alpha_t))\cdot \beta_t\\
	&\quad\quad \quad 
	 +\tilde{\e}\left(([\sigma_\mu]^\epsilon_{t}-\sigma_\mu(t,X_t,\LL(X_t),\alpha_t)(\tilde{X}_t))\cdot \tilde{Z}_t^{\alpha,\beta}\right)\Bigg\}dW_t.
	\end{align*}
	By It\^o formula, \eqref{X_lambda_C} and Assumption \ref{ass:pontryagin}, we get 
	\begin{align*}
	d\left (\frac{|\varDelta_t^\epsilon|^2}{2}\right )
	&\leq \bigg\{A_1 |\varDelta_t^\epsilon|^2 +\tilde{\e}\left(|[b_\mu]^\epsilon_{t}||\tilde{\varDelta}_t^\epsilon|\right)|\varDelta_t^\epsilon|\\
	&\quad \quad \quad 
	+|[b_x]^\epsilon_{t}-b_x(t,X_t,\LL(X_t),\alpha_t)||Z_t^{\alpha,\beta}||\varDelta_t^\epsilon|\\
	&\quad \quad \quad 
	+\epsilon|[b_\alpha]^\epsilon_{t}- b_\alpha(t,X_t,\LL(X_t),\alpha_t)||\beta_t||\varDelta_t^\epsilon|\\
	&\quad \quad \quad 
	+ \tilde{\e}\left(|[b_\mu]^\epsilon_{t}-b_\mu(t,X_t,\LL(X_t),\alpha_t)(\tilde{X}_t)||\tilde{Z}_t^{\alpha,\beta}|\right)|\varDelta_t^\epsilon|\bigg\}dt\\
	&\quad 
	+ \Bigg\langle \varDelta_t^\epsilon,\Bigg([\sigma_x]^\epsilon_{t}\cdot \varDelta_t^\epsilon+\tilde{\e}\left(\sigma_\mu^\epsilon(\varDelta)\cdot \tilde{\varDelta}_t^\epsilon\right) 
	+([\sigma_x]^\epsilon_{t}-\sigma_x(t,X_t,\LL(X_t),\alpha_t))\cdot Z_t^{\alpha,\beta}\\
	&\quad\quad \quad 
	 +\epsilon([\sigma_\alpha]^\epsilon_{t}- \sigma_\alpha(t,X_t,\LL(X_t),\alpha_t))\cdot \beta_t\\
	&\quad \quad \quad 
	+\tilde{\e}\left(([\sigma_\mu]^\epsilon_{t}-\sigma_\mu(t,X_t,\LL(X_t),\alpha_t)(\tilde{X}_t))\cdot \tilde{Z}_t^{\alpha,\beta}\right)\Bigg)dW_t\Bigg\rangle\\
	&\quad 
	+\frac52\bigg\{A_2|\varDelta_t^\epsilon|^2 +\Big(\int _0^1A_2(1+|X_t^{\lambda,\epsilon}|^2)d\lambda\Big)\tilde{\e}\left(|\tilde{\varDelta}_t^\epsilon|\right)^2\\
	&\quad \quad \quad 
	+|[\sigma_x]^\epsilon_{t}-\sigma_x(t,X_t,\LL(X_t),\alpha_t)|^2|Z_t^{\alpha,\beta}|^2\\
	&\quad \quad \quad 
	+\epsilon^2|[\sigma_\alpha]^\epsilon_{t}- \sigma_\alpha(t,X_t,\LL(X_t),\alpha_t)|^2|\beta_t|^2\\
	&\quad \quad \quad 
	+ \tilde{\e}\left(|[\sigma_\mu]^\epsilon_{t}-\sigma_\mu(t,X_t,\LL(X_t),\alpha_t)(\tilde{X}_t)||\tilde{Z}_t^{\alpha,\beta}|\right)^2\bigg\}dt.
	\end{align*}
	By Young Inequality, Jensen Inequality and assumption \ref{A1} we have 
	\begin{align*}
	\tilde{\e}\left(|[b_\mu]^\epsilon_{t}||\tilde{\varDelta}_t^\epsilon|\right)|\varDelta_t^\epsilon|&\leq \dfrac{1}{2}\left(\tilde{\e}\left(|[b_\mu]^\epsilon_{t}|^2|\tilde{\varDelta}_t^\epsilon|^2\right)+|\varDelta_t^\epsilon|^2\right)\\
	&\leq \dfrac{A_1}{2}\left(\int_{0}^{1}(1+|X_t^{\lambda,\epsilon}|^2)d\lambda\right )\tilde{\e}\left(|\tilde{\varDelta}_t^\epsilon|^2\right)+\frac12|\varDelta_t^\epsilon|^2.
	\end{align*}
	Since $\epsilon>0$ is chosen in a way that $\alpha+\epsilon\beta\in \mathbb{A}$, we can conclude by the a priori bound \eqref{moment_est} and the definition of $\mathbb{A}$, that 
	\begin{align*}
	\e\left(\sup_{s\in [0,t]}\tilde{\e}\left(|[b_\mu]^\epsilon_{s}||\tilde{\varDelta}_s^\epsilon|\right)|\varDelta_s^\epsilon|\right)
	\leq C(T,K,\|X_0\|_p)\tilde{\e}\left(\sup_{s\in [0,t]}|\tilde{\varDelta}_s^\epsilon|^2\right)
	+\e\left(\sup_{s\in [0,t]}\frac{|\varDelta_s^\epsilon|^2}{2}\right),
	\end{align*}
	for some constant $C(T,K,\|X_0\|_p)>0$ which does not depend on $\epsilon$. 
	 By the Burkholder-Davis-Gundy inequality, Young and Jensen inequalities we arrive at 
	\begin{align*}
	\e\left(\sup_{t\in [0,T]}|\varDelta_t^\epsilon|^2\right)
	&\leq I_1+I_2+I_3+I_4+I_5+I_6 + C\int_{0}^{T}\e\left(\sup_{s\in [0,t]}|\varDelta_s^\epsilon|^2\right)ds,
	\end{align*}
	for a constant $C>0$ which does not depend on $\epsilon$ and 
	\begin{align*}
	I_1&=\e\left(\int_{0}^{T}|[b_x]^\epsilon_{t}-b_x(t,X_t,\LL(X_t),\alpha_t)|^2|Z_t^{\alpha,\beta}|^2dt\right)\\
	I_2&=\epsilon^2\e\left(\int_{0}^{T}|[b_\alpha]^\epsilon_{t}-b_\alpha(t,X_t,\LL(X_t),\alpha_t)|^2|\beta_t|^2dt\right)\\
	I_3&=\e\left(\int_{0}^{T}\tilde{\e}\left(|[b_\mu]^\epsilon_{t}-b_\mu(t,X_t,\LL(X_t),\alpha_t)(\tilde{X}_t)|^2|\tilde{Z}_t^{\alpha,\beta}|^2\right)dt\right)
	\end{align*}
	and $I_4,I_5,I_6$ are analogues for $\sigma$. We will only show $I_1\rightarrow 0$ as $\epsilon\rightarrow 0$, the other terms being handled by similar arguments. By assumption \ref{A1} we have
	\begin{align*}
	|[b_x]^\epsilon_{t}-b_x(t,X_t,\LL(X_t),\alpha_t)|^4\leq C(1+|X_t|^{4q-4}+|X_t^{\lambda,\epsilon}|^{4q-4}).
	\end{align*}
	Furthermore we have for any $p\leq r$ that
	\begin{align*}
	\e\left(\sup_{0\leq t\leq T}|X_t^{\lambda,\epsilon}|^p\right)&\leq C_p\Bigg\{(1+\lambda^p) \e\left(\sup_{0\leq t\leq T}|X_t^{\epsilon}|^p\right)+\lambda^p\e\left(\sup_{0\leq t\leq T}|X_t|^p\right)\Bigg\},
	\end{align*}
	is bounded from above by some constant that does not depend on $\epsilon$ for $\epsilon>0$ small enough. Since  $X^{\lambda,\epsilon}\rightarrow X$ in $L^2(\Omega;C([0,T];\R^d))$, by the a-priori bound \eqref{moment_est}, the estimate $\e\left(\sup_{t\in [0,T]}|Z_t|^4\right)<\infty$, the continuity of $b_x$ and the dominated convergence theorem, one concludes that $I_1\rightarrow 0$ as $\epsilon\rightarrow 0$. Similar arguments combined with Gronwall's lemma finish the proof.
\end{proof}

As an important consequence, we obtain the following formula for the G\^ateaux derivative of the cost functional.
Given Lemma \ref{lem:variation}, the next result is proven in the same way as its done in \cite{carmona2018probabilistic} and thus omitted.

\begin{corollary}
	The cost functional
	\begin{align*}
	J\colon\A\rightarrow \R
	\end{align*}
	is G\^ateaux differentiable and its G\^ateaux derivative at $\alpha\in \A$ in direction $\beta\in \A$ is given by
	\begin{align*}
	dJ(\alpha)\cdot \beta
	&=\e\left(f_x(t,X_t,\LL(X_t),\alpha_t)\cdot Z_t^{\alpha,\beta}+f_\alpha(t,X_t,\LL(X_t),\alpha_t)\cdot \beta_t\right)\\
	&\quad +\e\left(\tilde{\e}\left(f_\mu(t,X_t,\LL(X_t),\alpha_t)(\tilde{X}_t)\cdot \tilde{Z}_t^{\alpha,\beta}\right)\right)\\
	&\quad +  \e\left(g_x(X_T,\LL(X_T))\cdot Z_t^{\alpha,\beta}+\tilde{\e}\left( g_\mu(X_T,\LL(X_T))(\tilde{X}_T)\cdot \tilde{Z}_T^{\alpha,\beta}\right)\right).
	\end{align*}
\end{corollary}

\subsection{Maximum Principle}

%For simplicity we will assume, that $\sigma_\mu(t,x,\mu,\alpha)(y)=0$. In the example of the FitzHugh-Nagumo model this is the case, when the maximum conductance for the synaptic transmission is assumed to be deterministic.
For the reader's convenience, we now rewrite the adjoint equation of section \ref{sec:intro} using Hamiltonian formalism. Recall that for $x,y,p\in\R^d,$ $q\in\R^{d\times m}$ $\mu\in \mathcal P_2$ and $\alpha\in A,$ we introduced the quantity
\begin{align*}
H(t,x,\mu,p,q,\alpha):=\langle b(t,x,\mu,\alpha),p\rangle+\langle\sigma(t,x,\mu,\alpha), q\rangle +f(t,x,\mu,\alpha)\,.
\end{align*}
Thus, given a control $\alpha\in\A,$ one sees that the pair $(P,Q)\in \mathcal{S}^{2,d}\times \mathcal{H}^{2,d\times m}$ solves the adjoint equation if and only if for all $t\in [0,T]$, $\p$-almost surely
\begin{equation}\label{adjoint_H}
\left \{\begin{aligned}
	dP_t&=-\left [H_x(t,X_t,\LL(X_t),P_t,Q_t,\alpha_t) + \mathbb{\tilde E}\left (H_\mu(t,X_t,\LL(X_t),\tilde P_t,\tilde Q_t,\alpha_t)(\tilde X_t)\right )\right ]dt
	+Q_tdW_t
	\\
	P_T&=g_x(X_T,\LL(X_T))+\tilde{\e}\left(g_\mu(X_t,\LL(X_T))(\tilde X_T)\right).
\end{aligned}\right .
\end{equation}
where $(\tilde X,\tilde P,\tilde Q,\tilde \alpha)$ is an independent copy of $(X,P,Q,\alpha)$ on the space $(\tilde\Omega,\mathcal{\tilde F},\mathbb{\tilde P}).$

Let us point out that the above coefficients fail to satisfy \cite[Assumption MKV SDE, Chap.~4]{carmona2018probabilistic}.
Hence, we first need to address the solvability of the BSDE \eqref{adjoint_H} under the assumptions of Theorem \ref{thm:max}. 
\begin{lemma}
	\label{lem:adjoint}
	Under the assumptions of Theorem \ref{thm:max},
there exists a unique solution $(P,Q)\in \mathcal{S}^2\times \mathcal{H}^{2,d\times m}$ of \eqref{adjoint_H}.
\end{lemma}
\begin{proof}
	Fix $\alpha\in \A$ and for simplicity, denote by $H_x(t,\omega,p,q):=H_x(t,X_t(\omega),\mathcal{L}(X_t),p,q,\alpha_t)$ and by $H_\mu(t,\omega,x,p,q):=H_\mu(t,x,\mathcal{L}(X_t),p,q,\alpha_t)(X_t(\omega)).$
	Consider the map $\Gamma:\mathcal{H}^{2,d}\times \mathcal{H}^{2,d\times m}\rightarrow \mathcal{H}^{2,d}\times \mathcal{H}^{2,d\times m}$ which maps a given pair \[ (Y,Z)\in \mathcal{H}^{2,d}\times \mathcal{H}^{2,d\times m} \]
	to the solution $(P,Q)$ of 
	\begin{equation}\label{upper}
	\left \{\begin{aligned}
	dP_t&=-\left[ H_x(t,\omega,P_t,Q_t)+\e\left(H_\mu(t,\omega,X_t,Y_t,Z_t)\right)\right]dt+ QdW_t
	\\
	P_T&=g_x(X_T(\omega),\LL(X_T))+\tilde{\e}\left(g_\mu(\tilde{X}_t,\LL(X_T))(X_T)\right),
	\end{aligned}\right .
	\end{equation}
	where the expectation is to be understood in the following way: 
	\begin{align*}
	\e\left(H_\mu(t,\omega,X_t,Y_t,Z_t)\right)=\int_{\Omega}H_\mu(t,\omega,X_t(\omega'),Y_t(\omega'),Z_t(\omega'))\p(d\omega').
	\end{align*}
	In the following we drop the dependence on $\omega$ for $H_\mu$.
	
	Since the above equation is a standard backward SDE with monotone coefficients, the existence of a solution is well-known by standard results. 
	We will now show that the map $\Gamma$ is a contraction, when the space $\mathcal{H}^{2,d}\times \mathcal{H}^{2,d\times m}$ is equipped with the norm 
	\[ \interleave (P,Q)\interleave_\gamma :=\left(\int_{0}^{T}e^{\gamma t}(\|P_t\|_2^2+\|Q_t\|_2^2)dt\right )^{1/2} ,\] for a sufficiently large parameter $\gamma>0$. If we denote by $(P^1,Q^1),(P^2,Q^2)$ two solutions of \eqref{upper} for $(Y^1,Z^1)$ and $(Y^2,Z^2)$ respectively, then by the backward It\^o Formula \cite[p.\ 356]{pardoux2014stochastic} applied to $e^{\gamma t}|P^1_t-P^2_t|^2$ we get 
	\begin{multline}
	\label{ito_backward}
	|P^1_t-P^2_t|^2 
	+\e\left(\int_{t}^{T}\gamma e^{\gamma (r-t)}|P^1_r-P^2_r|^2dr\enskip\Bigg|\enskip\mathcal{F}_t\right)+\e\left(\int_{t}^{T} e^{\gamma (r-t)}|Q^1_r-Q^2_r|^2dr\enskip\Bigg|\enskip\mathcal{F}_t\right)
	\\
	\leq 2\e\Bigg(\int_{t}^{T}e^{\gamma (r-t)}\Bigg\{\big( H_x(t,P^1_r,Q^1_r)- H_x(r,P^2_r,Q^2_r)\big)\cdot(P^1_r-P^2_r) +|P^1_r-P^2_r|\times
	\\ \int_{\Omega}|H_\mu(r,X_r(\omega'),Y_r^1(\omega'),Z_r^1(\omega'))-H_\mu(r,X_r(\omega'),Y_r^2(\omega'),Z_r^2(\omega'))|\p(d\omega')\Bigg\}dr\enskip\Bigg|\enskip\mathcal{F}_t\Bigg).
	\end{multline}
	From assumptions \ref{A1},\ref{A2}, Young's inequality and Lemma \ref{lem:constraint}, we infer that
	\begin{align*}
	&\|(H_x(t,\omega,P^1,Q^1)
	-H_x(t,\omega,P^2,Q^2))\cdot (P^1-P^2)\|_1
	\\
	&\quad \quad 
	\leq (A_1+A_2^2)
	\|P^1-P^2\|_2^2+\frac{1}{4}\|Q^1-Q^2\|_2^2
	\intertext{and}
	&\int_{\Omega}|H_\mu(r,X_r(\omega'),Y_r^1(\omega'),Z_r^1(\omega'))-H_\mu(r,X_r(\omega'),Y_r^2(\omega'),Z_r^2(\omega'))|\p(d\omega')
	\\
	&\quad \quad 
	\leq (A_1\vee A_2)(1+\|X_r\|_2^2)^\frac{1}{2}(\|Y_r^1-Y_r^2\|_2+\|Z_r^1-Z_r^2\|_2)\,.
	\end{align*}
%	Thus, using Jensen inequality, Cauchy-Schwarz Inequality and Young Inequality, one sees that the right hand side of \eqref{ito_backward} is estimated above in $L^1(\Omega)$ by the quantity
%	\begin{align*}
%	\int_{t}^{T}e^{\gamma (r-t)}\left \{\frac{1}{2}\|P^1_r-P^2_r\|_2^2
%	+C_A(1+\|X_r\|_2^2)\left (\|Y_r^1-Y_r^2\|_2^2+\|Z_r^1-Z_r^2\|_2^2\right )\right \}dr,
%	\end{align*}
%	for the latter part of the above equation, for some constant $C_A>0$. 
	Invoking \eqref{moment_est}, Cauchy-Schwarz and Young Inequalities, we can conclude that
	\begin{align*}
	&\int_{0}^{T}\gamma e^{\gamma r}\|P^1_r-P^2_r\|_2^2dr+\int_{0}^{T} e^{\gamma r}\|Q^1_r-Q^2_r\|_2^2dr\\
	&\leq 2(A_1+A_2^2)\int_{0}^{T}e^{\gamma r}\|P_r^1-P_r^2\|_2^2dr+\frac{1}{2}\int_{0}^{T}e^{\gamma r}\|Q_r^1-Q_r^2\|^2_2dr\\
	&\quad + \frac{1}{2}\int_{0}^{T}e^{\gamma r}\|P^1_r-P^2_r\|_2^2dr
	+C(A,\|X_0\|_2,K)\int_{0}^{T}e^{\gamma r}\left(\|Y_r^1-Y_r^2\|_2^2+\|Z_r^1-Z_r^2\|_2^2\right)dr.
	\end{align*}
	For $\gamma$ large enough this leads to 
	\begin{align*}
	\interleave (P^1-P^2,Q^1-Q^2)\interleave_{\gamma} ^2
	\leq \frac{1}{2} \interleave(Y^1_r-Y^2_r,Z^1_r-Z^2_r)\interleave_\gamma ^2\,,
	\end{align*}
	showing that $\Gamma$ is a contraction. The conclusion follows.
\end{proof}

The following corollary follows immediately by integration by parts and an application of Fubini Theorem. We therefore omit the proof and refer to \cite[Lemma.~6.12 p.\ 547]{carmona2018probabilistic}.
\begin{corollary}
	Let $(P,Q)$ be a solution to \eqref{adjoint_H}, then it holds
	\begin{multline}\label{duality}
	\e\left(\langle P_T,Z_T^{\alpha,\beta}\rangle \right)
	=\e\left(\int_{0}^{T}\langle P_t,b_\alpha(t,X_t,\LL(X_t),\alpha_t)\cdot \beta\rangle
	+\langle Q_t,\sigma_\alpha(t,X_t,\LL(X_t),\alpha_t)\cdot \beta\rangle dt\right)
	\\
	- \e\left(\int_{0}^{T}f_x(t,X_t,\LL(X_t),\alpha_t)\cdot Z_t^{\alpha,\beta}+\tilde{\e}\left(f_\mu(t,X_t,\LL(X_t),\alpha_t)(\tilde{X}_t)\cdot \tilde{Z}_t^{\alpha,\beta}\right)\right).
	\end{multline}
\end{corollary}

\begin{remark}
	An immediate consequence of \eqref{duality} is the following formula for the G\^ateaux derivative of the cost functional
	\begin{align*}
	d J(\alpha)\cdot\beta	&=\e\bigg(\int_{0}^{T}\Big\{\langle b_\alpha(t,X_t,\LL(X_t),\alpha_t)\cdot \beta_t,P_t\rangle
	+\langle\sigma_\alpha(t,X_t,\LL(X_t),\alpha_t)\cdot \beta_t,Q_t\rangle
	\\
	&\quad \quad \quad \quad 
	+f_\alpha(t,X_t,\LL(X_t),\alpha_t)\cdot \beta_t\Big\}dt \bigg)
	\\
	&=\e\bigg(\int_{0}^{T} H_\alpha(t,X_t,\mathcal{L}(X_t),P_t,Q_t,\alpha_t)\cdot \beta_tdt\bigg).
	\end{align*}
	An application of Fubini Theorem then leads to the following representation for the gradient of $J$:
	\begin{align}
	\label{gradient_J}
	\nabla J(\alpha)_t
	=\e\left( H_\alpha(t,X_t,\mathcal{L}(X_t),P_t,Q_t,\alpha_t)\right)\,,\quad t\in [0,T].
	\end{align}
	It is hardly necessary to mention that the formula \eqref{gradient_J} is of fundamental importance for numerical purposes, see Section \ref{sec:numerics} below.
\end{remark}

We are now in position to prove the maximum principle.

\begin{proof}[Proof of Theorem \ref{thm:max}]
Let $\alpha\in \mathbb{A}$ be an optimal control for \eqref{SM}, $X$ the corresponding solution to (\ref{state}) and $(P,Q)$ the associated solution to (\ref{adjoint_H}). For $\beta\in \mathbb{A}$ we have by the optimality of $\alpha$
\begin{align*}
dJ(\alpha)\cdot (\beta-\alpha)=\langle \nabla J(\alpha),\beta-\alpha\rangle_{L^2([0,T];\R^k)}\geq 0\,.
\end{align*}
Invoking the convexity of the Hamiltonian (see Assumption \ref{ass:weak}), we get
\begin{align*}
\int_{0}^{T}\e\Big(H(t,X_t,\mathcal{L}(X_t),P_t,Q_t,\alpha_t)-H(t,X_t,\mathcal{L}(X_t),P_t,Q_t,\beta_t)\Big)dt\geq 0\,.
\end{align*}
For any arbitrary measurable set $C\subset[0,T]$ and $\alpha\in A$ we can define the admissible control
\begin{align*}
\beta_t=\begin{cases}
\alpha & \text{for}\enskip t\in C,\\
\alpha_t & \text{otherwise},
\end{cases}
\end{align*}
hence 
\begin{align*}
\int_{0}^{T}\mathbf{1}_C(t)\e\big(H(t,X_t,\mathcal{L}(X_t),P_t,Q_t,\alpha_t)-H(t,X_t,\mathcal{L}(X_t),P_t,Q_t,\alpha)\big)dt\geq 0.
\end{align*}
Therefore we get 
\begin{align*}
\e\big(H(t,X_t,\mathcal{L}(X_t),P_t,Q_t,\alpha_t)-H(t,X_t,\mathcal{L}(X_t),P_t,Q_t,\alpha)\big)\geq 0,
\end{align*}
$dt$-almost everywhere. This proves the theorem.
\end{proof}

\section{Numerical examples}
\label{sec:numerics}

In this section we focus on the FitzHugh-Nagumo model with external noise only, i.e.\ the system of $3N$ stochastic differential equations:
\begin{equation}\label{state_num}
\left\{
\begin{aligned}
&dv^i_t=\Big(v^i_t-\frac{(v^ i_t)^3}{3}-w_t^i + \alpha_t - \frac{1}{N}\sum\nolimits_{j=1}^N\bar J(v^i_t-V_{rev})y^j_t\Big)dt  +\sigma _{ext} d W^i_t
\\
&dw^i_t=c (v^i_t+a -b w^i_t)dt\enskip ,
\\
&dy^i_t=(a_r S (v_t^i)(1-y_t^i)-a_d y_t^i)dt ,
\end{aligned}\right.
\end{equation}
where we recall that
$S(v):= T_{max}/[1+\exp\lambda(v-V_T)].$

We are interested in controlling the average membrane potential (called in the following ``local field potential'') of a network of FitzHugh-Nagumo neurons into a desired state. Our cost functional is given by
\begin{equation}\label{choice:cost}
\begin{aligned}
f(t,x,\mu,\alpha)&:=|\int_{\R^3}v\mu(dv\times dw\times dy)-\overline{v}_t|^2
\\
g(t,x)&:=0,
\end{aligned}
\end{equation}
where $(\overline{v}_t)_t$ is a certain reference profile. We should mention that the average membrane potential will only give an idea about the average activity of the network at each time. For example a high average membrane potential is an indication that a high number of neurons are in the regenerative or active phase, while a low average membrane potential means that a high number of neurons are in the absolute refractory or silent phase.\smallskip

In the described case the adjoint equation is reduced to 
\begin{equation}\label{adjoint}
\left \{\begin{aligned}
dP_t&=-\bigg \{\left \langle b_x(t,X_t,\LL(X_t),\alpha_t),P_t\right \rangle
 + \mathbb{\tilde E}\left(\left\langle b_\mu(t,X_t,\LL(X_t),\alpha_t)(\tilde{X}_t),\tilde{P}_t\right\rangle\right )
 \\
&\quad \quad \quad 
 +\mathbb{\tilde E}\left (f_\mu(t,X_t,\LL(X_t),\alpha_t)(\tilde{X}_t)\right )\bigg \}dt  +Q_tdW_t
\\
P_T&=0.
\end{aligned}\right .
\end{equation}
In the following section we will give a short introduction on how to solve \eqref{adjoint} numerically.

\subsection{Numerical approximation of the adjoint equation} 

In general we consider the following non fully coupled MFFBSDE

\begin{equation}
\label{FBSDE}
\left\{
\begin{aligned}
&dX_t=b(t,X_t,\mathcal{L}(X_t))dt+\sigma(t,X_t,\mathcal{L}(X_t))dW_t
\\
&dY_t= \left[f(t,X_t,Y_t)+h(t,X_t,\mathcal{L}(X_t,Y_t))\right]dt-Z_tdW_t\enskip
\\
&X_0=\xi\\
&Y_T=g(X_T).
\end{aligned}\right.
\end{equation}
For the approximation of the forward component we consider an implicit Euler scheme for McKean-vlasov equations. Since this is standard, we will not go into further details. Concerning the backward component, we consider a scheme similar to the one presented in \cite{num}. We should mention that since we are not dealing with a fully coupled MFFBSDE, our situation is a lot easier to handle than the one treated in \cite{num}. For a given discrete time grid $\pi:0=t_0<t_1<...<t_N=T$, we consider the following numerical scheme:
\begin{align*}
Y_{t_k}^{\pi}&=\e\left(Y_{t_k+1}^{\pi}|\mathcal{F}_{t_k}\right)-(t_{k+1}-t_{k})\bigg\{f(t_k,X_{t_k}^{\pi},Y_{t_k}^{\pi})+h(t_{k+1},X_{t_{k}}^{\pi},\mathcal{L}(X_{t_{k+1}}^\pi,Y_{t_{k+1}}^\pi))\bigg\}\\
Z_{t_k}^{\pi}&:=(t_{k+1}-t_{k})^{-1}\e\left(Y_{t_k}^{\pi}(W_{t_{k+1}}-W_{t_{k}})|\mathcal{F}_{t_k}\right),\\
Y_{t_n}^{\pi}&=g(X_{t_n}^{\pi}),\quad Z_{t_n}^{\pi}=0.
\end{align*}
For the approximation of the conditional expectation, we make use of the decoupling field mentioned in \cite{carmona2018probabilistic}, to write
\begin{align*}
Y_{t_k+1}^{\pi}=u(t_{k+1},X_{t_{k+1}}^{\pi},\mathcal{L}(X_{t_{k+1}}^{\pi}))=:u(t_{k+1},X_{t_{k+1}}^{\pi}).
\end{align*}
Thus we can represent the conditional expectation in terms of a function $\tilde{u}$ by
\begin{align*}
\e\left(Y_{t_k+1}^{\pi}|\mathcal{F}_{t_k}\right)=\tilde{u}(t_{k+1},X_{t_k}^{\pi}).
\end{align*}
We approximate $\tilde{u}(t_{k+1},\cdot)$ with gaussian radial basis functions, by solving the following minimization problem for fixed nodes $x_1,...,x_L$:
\begin{align*}
\min_{\alpha}\e\left(|Y_{t_k+1}^{\pi}-\sum_{i=1}^{L}\alpha_i(t_{k+1})e^{\frac{1}{2 \delta}\|X_{t_k}^\pi-x_i\|^2}|^2\right),
\end{align*}
for $\alpha=(\alpha_1(t_{k+1}),...,\alpha_L(t_{k+1}))^\dagger$, where $\delta>0$ and $L\in \mathbb{N}$ are fixed. Therefore we initialize our reference points $x_1,...,x_L$ by $L$ independent realizations of $X_{t_k}^\pi$. For $M$ realizations of $Y_{t_k+1}^{\pi}$ and $X_{t_k}^\pi$, denoted  by $y_{t_k+1}^1,...,y_{t_k+1}^m$ and $x_{t_k+1}^1,...,x_{t_k+1}^m$ respectively, we then write 
\begin{align*}
y_{t_k+1}&=(y_{t_k+1}^1,...,y_{t_k+1}^m)^\dagger\\
A(t_k)&=(e^{\frac{1}{2 \delta}\|x_{t_k}^i-x_j\|^2})_{i=1,...,m,j=1,...,L}.
\end{align*}
Thus we need to minimize
\begin{align*}
\|y_{t_k+1}-A(t_k)\alpha(t_{k+1})\|^2.
\end{align*} 
A similar approach for BSDEs can be found in \cite{bsde}. There is no convergence analysis of this scheme for our assumptions on the coefficients, this should only give an idea how to solve the adjoint equation in practice. Furthermore we should mention, that in the case where only external noise is present, the duality (\ref{duality}) and the resulting gradient representation still holds true for any non adapted solution of (\ref{adjoint_H}). Thus one can also implement a numerical scheme for the adjoint equation, without any conditional expectations involved.

\subsection{Gradient descent algorithm}

We will now briefly sketch our gradient decent algorithm.
\begin{algorithm}
	Take an initial control $\alpha _0\in\mathbb A$, $s_0>0$, and recursively for $n=0,1,\dots:$
	\begin{itemize}
		\item[-] determine $X^{\alpha _n}$ by solving the state equation with an implicit particle scheme to avoid particle corruption;
		\item[-] solve the adjoint equation for given $X^{\alpha _n}$ in order to approximate $(P^{\alpha_n},Q^{\alpha_n})$;
		\item[-] approximate the gradient
		\[
		\nabla J(\alpha _n)_s
		=\e\Big[\langle b_\alpha(s,X_s^{\alpha _n},\LL (X^{\alpha _n}_s),\alpha _s^n),P^{\alpha _n}_s\rangle
		+ f_\alpha(s,X_s^{\alpha _n},\LL ( X_s^{\alpha _n}),\alpha ^n_s)
		\Big]
		\]
		via Monte-Carlo method, where $(P^{\alpha_n},Q^{\alpha_n})$ solves the adjoint equation;
		\item[-] update the control in direction of the steepest decent: $\alpha _{n+1}:=\alpha _n -s_n\nabla J(\alpha _n)$;
		\item[-] accept the new control if the cost corresponding to the new control is smaller than the previous cost, otherwise decrease the step size: $s_n=s_n/2$ and go back to step 2
		\item[-] the algorithm stops if $\|\nabla J(\alpha_n)\|<\epsilon$
	\end{itemize}
\end{algorithm}

To compute the expectation term, one is in fact reduced to simulate the solution of the network equation itself and use the particles as samples for the Monte-Carlo simulation.

\subsection{Numerical examples for systems of FitzHugh-Nagumo Neurons}

Although the solution to the adjoint equation is a $3$-dimensional process, in the following we will only plot its first variable, since the other variables are irrelevant for the gradient in our situation.\smallskip

To illustrate some problems we had with the simulations, we consider the example of the deterministic uncoupled case of equation (\ref{state_num}), where $\overline{J}=0$ and $\sigma_{ext}=0$. 
In the given situation the membrane potential $v$ becomes highly sensitive to small perturbations of the control at specific times, when we chose the control $\alpha_t\equiv \alpha$ close to the bifurcation value for the supercritical Hopf bifurcation point of the equation. This sensitivity can lead to high valued solutions of the corresponding adjoint equation for specific reference profiles. One example is to choose the reference profile as the $v$-trajectory of a solution to (\ref{state_num}), for a control parameter $\alpha$ in the limit cycle regime. This situation is illustrated by the figures below.

\begin{figure}[H]
	\begin{minipage}[t]{.32\linewidth} % [b] => Ausrichtung an \caption
		\centering
		\includegraphics[width=\linewidth]{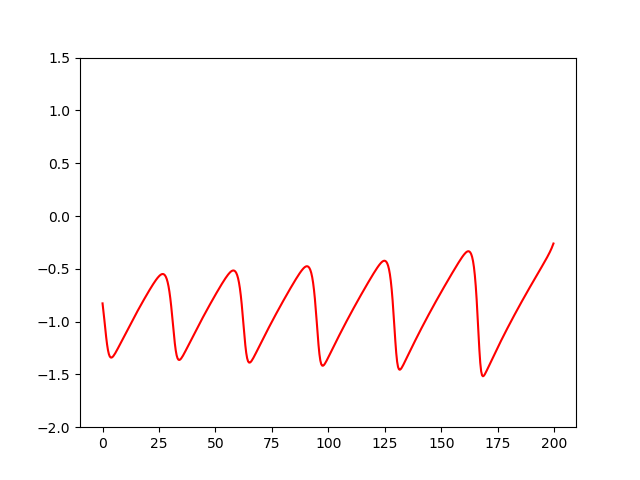}
		\caption{Membrane potential of the solution to (\ref{state_num}) for $\alpha\equiv 0.315$}
	\end{minipage}
	\hspace{.015\linewidth}% Abstand zwischen Bilder
	\begin{minipage}[t]{.32\linewidth} % [b] => Ausrichtung an \caption
		\centering
		\includegraphics[width=\linewidth]{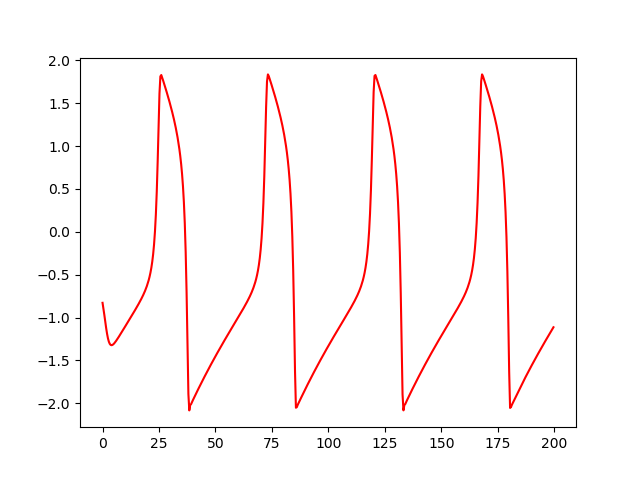}
		\caption{Reference profile generated by solving (\ref{state_num}) for $\alpha\equiv 0.33$}
	\end{minipage}
	\hspace{.015\linewidth}% Abstand zwischen Bilder
	\begin{minipage}[t]{.32\linewidth} % [b] => Ausrichtung an \caption
		\centering
		\includegraphics[width=\linewidth]{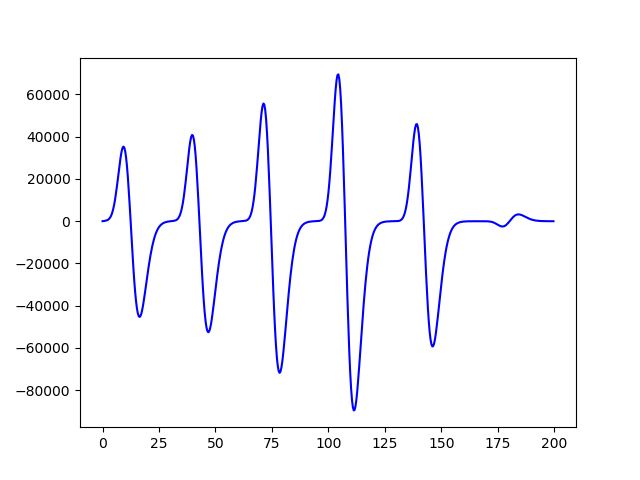}
		\caption{Solution to the corresponding adjoint equation}
	\end{minipage}
\end{figure}

The same type of phenomena also occurs in the case of the coupled system of stochastic FitzHugh-Nagumo neurons. Here it can lead to high fluctuations of the sample mean for the adjoint equation, thus a high number of particles is required to compute the expectation of the solution to the adjoint equation. A small illiustration is given by the figures below. 

\begin{figure}[H]
	\begin{minipage}[t]{.32\linewidth} % [b] => Ausrichtung an \caption
		\centering
		\includegraphics[width=\linewidth]{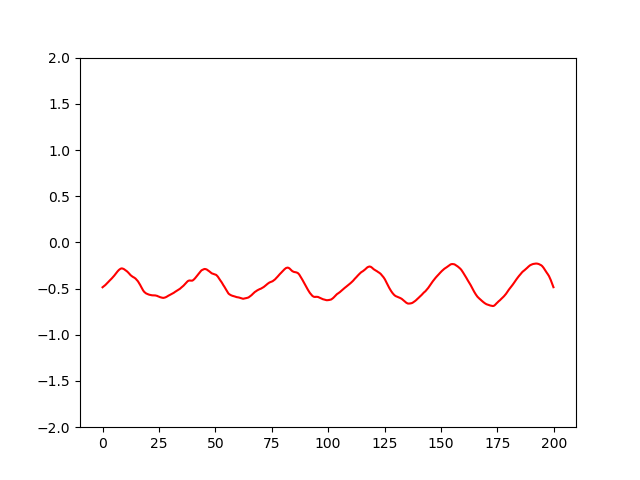}
		\caption{Local field potential of the solution to (\ref{state_num}) for $\alpha\equiv 0$}
	\end{minipage}
	\hspace{.015\linewidth}% Abstand zwischen Bilder
	\begin{minipage}[t]{.32\linewidth} % [b] => Ausrichtung an \caption
		\centering
		\includegraphics[width=\linewidth]{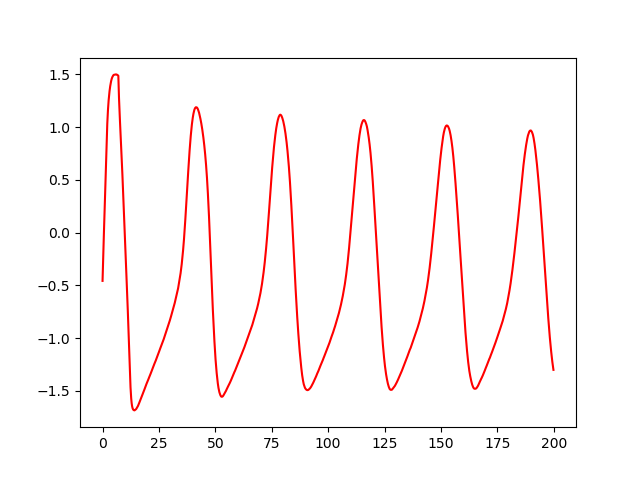}
		\caption{Reference profile chosen as the local field potential of \eqref{state_num} for $\alpha(t)= 0.8$  if $t\leq 7$}
	\end{minipage}
	\hspace{.015\linewidth}% Abstand zwischen Bilder
	\begin{minipage}[t]{.32\linewidth} % [b] => Ausrichtung an \caption
		\centering
		\includegraphics[width=\linewidth]{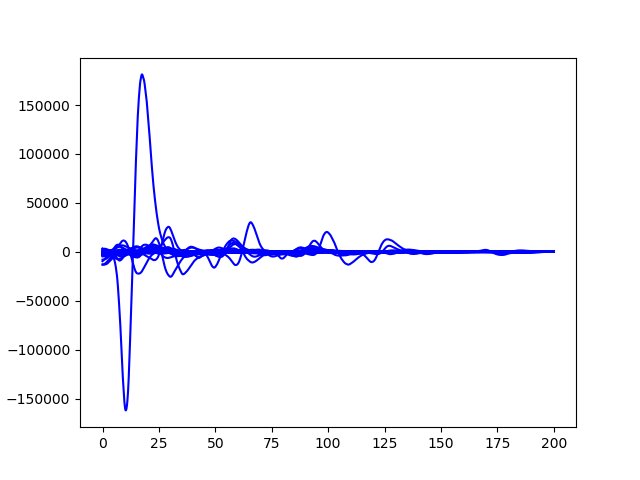}
		\caption{Samples of the solution to the adjoint equation}
	\end{minipage}
\end{figure}

In this example and in the following, the initial states are uniformly distributed on the orbit of a solution to (\ref{state_num}) with $\alpha\equiv 0$, $\sigma_{ext}=0$ and initial conditions $V_0=-0.828,w_0=-0.139,y_0=0.589$. The other parameters are given below in Table 1. Furthermore we are always using $N=1000$ particles for the particle approximation of (\ref{state_num}).

%\begin{example}[Control of a coupled system of FitzHugh-Nagumo Neurons]
\subsubsection{Control of a coupled system of FitzHugh-Nagumo Neurons}
	For our first example, we consider a parameter regime where the activity of a large number of neurons of the network at some time $t$ leads to further activity at a later time, without any external current applied to the system. Therefore we slow down the gating variable, by decreasing the closing rate of the synaptic gates. This way its impact on the network is still high enough, when a large part of the network is excitable again.\smallskip

	Our goal is now to increase the activity of the network up to time $t=100$ and then control the network back into its resting potential. Up to time $t=100$, the following reference profile shows the local field potential of a network of coupled FitzHugh-Nagumo neurons, when a constant input current of magnitude $0.8$ is applied for a time period of $\Delta t=7$ at $t=0$. For times $t>100$ it shows the resting potential of a single FitzHugh-Nagumo neuron.
	
	\begin{figure}[H]
		\begin{minipage}[t]{.45\linewidth} % [b] => Ausrichtung an \caption
			\centering
			\includegraphics[width=\linewidth]{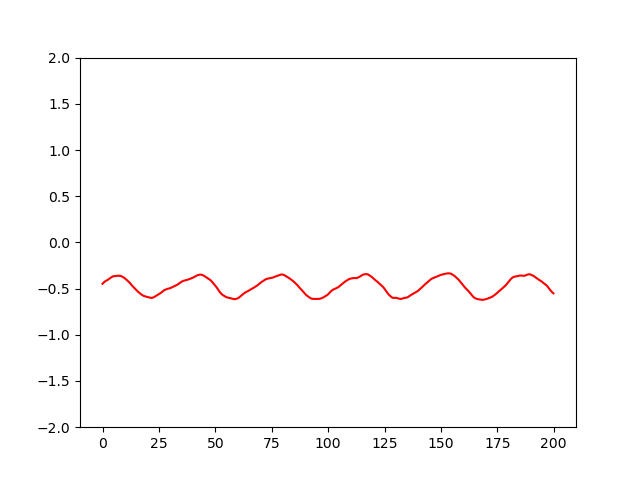}
			\caption{Uncontrolled local field potential}
		\end{minipage}
		\hspace{.1\linewidth}% Abstand zwischen Bilder
		\begin{minipage}[t]{.45\linewidth} % [b] => Ausrichtung an \caption
			\centering
			\includegraphics[width=\linewidth]{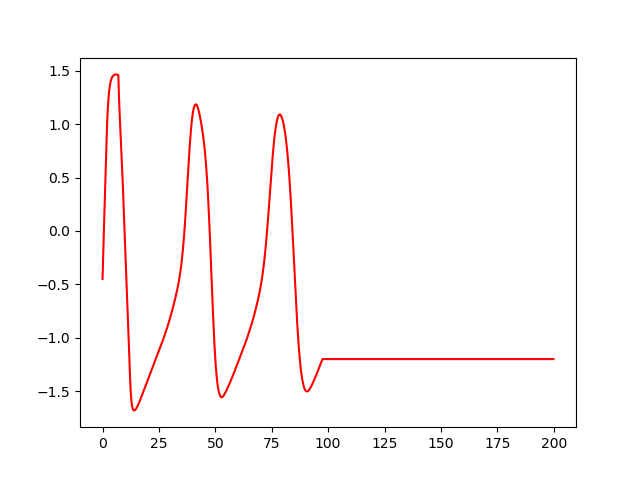}
			\caption{Reference profile}
		\end{minipage}
	\end{figure}

	We expect the optimal control to raise the membrane potential for a small time period at $t=0$ and then counteract the stimulating effect of the coupling around $t=100$. However this effects should not occur in the uncoupled setting, which we will consider afterwards. 
	
	The following shows the optimal control and the corresponding optimal local field potential. We remind that this might only be locally optimal, since we cannot expect to find a globally optimal control with our gradient decent algorithm.
	
	\begin{figure}[H]
		\begin{minipage}[t]{.45\linewidth} % [b] => Ausrichtung an \caption
			\centering
			\includegraphics[width=\linewidth]{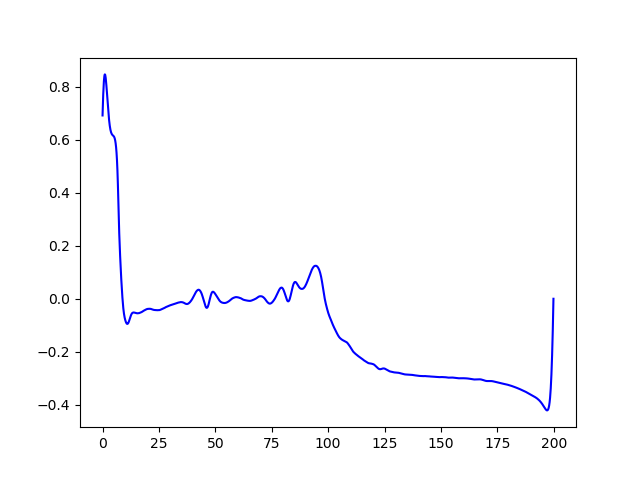}
			\caption{Optimal control}
		\end{minipage}
		\hspace{.1\linewidth}% Abstand zwischen Bilder
		\begin{minipage}[t]{.45\linewidth} % [b] => Ausrichtung an \caption
			\centering
			\includegraphics[width=\linewidth]{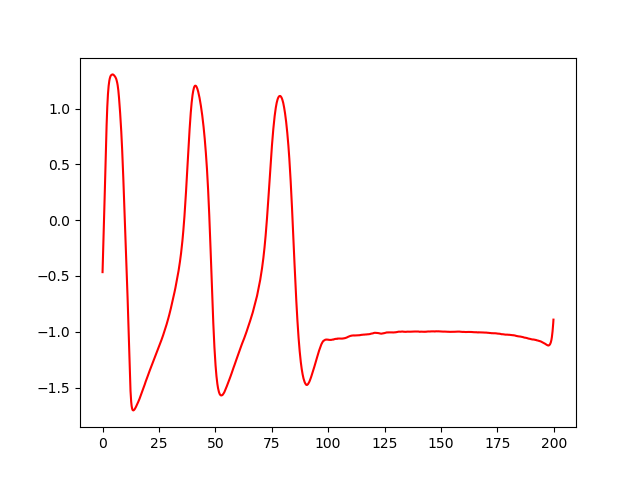}
			\caption{Local field potential with optimal control}
		\end{minipage}
	\end{figure}

\newpage
	
%\end{example}

%\begin{example}[Control of a uncoupled system of FitzHugh-Nagumo Neurons]
\subsubsection{Control of an uncoupled system of FitzHugh-Nagumo Neurons}
	Now we investigate the control problem for the uncoupled equation (\ref{state_num}), where $J=0$. Since the reference profile it still the same as in example 5.3.1, we will only present the corresponding optimal control.

	\begin{figure}[H]
		\begin{minipage}[h]{.45\linewidth} % [b] => Ausrichtung an \caption
			\includegraphics[width=\linewidth]{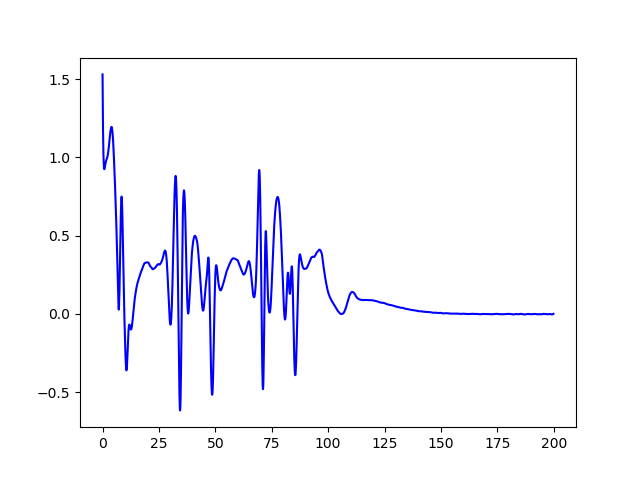}
			\caption{Optimal control}
		\end{minipage}
	\end{figure}

	As expected, the control does not need to counteract any stimulating effects for times $t>100$. Furthermore it is not sufficient in the uncoupled case to apply an input current for a small time period at $t=0$, to reach the desired local field potential up to time $t=100$. 
	
%\end{example}

\begin{center}
	\begin{tabular}{ |p{3cm}|p{3cm}|p{3cm}|  }
		\hline
		\multicolumn{3}{|c|}{\textbf{Table 1}: Parameters used for the examples} \\
		\hline
		Time parameters & FitzHugh-Nagumo parameters & Synapse\\
		\hline
		$t_{end}=200$ & $a=0.7$  &  $V_{rev}=1$  
		\\
		$\Delta t=0.1$ & $b=0.8$  & $a_r=1$
		\\
		& $c=0.08$ & $a_d=0.3$ 
		\\
		& $\sigma_{ext}=0.04$  &  $T_{max}=1$
		\\
		&   & $\lambda=0.1$
		\\
		&   & $V_{rev}=1.2$
		\\
		&  & $V_T=2$
		\\
		&  & $J=0.46$
		\\
		&  & $\sigma_J=0$
		\\
		\hline
	\end{tabular}
\end{center}

\section*{Acknowledgement} This work has been funded by Deutsche Forschungsgemeinschaft (DFG) through grant CRC 910 ``Control of self-organizing nonlinear systems: Theoretical methods and concepts of application'', Project (A10) ``Control of stochastic mean-field equations with applications to brain networks.''
Both authors are thankful to Wilhelm Stannat, Tilo Schwalger and Fran\c{c}ois Delarue for helpful discussions.

%\bibliography{biblio2}
\bibliographystyle{abbrv}

\end{document}